\documentclass[10pt,leqno]{amsart}

\usepackage[normalem]{ulem}

\usepackage{palatino,mathrsfs}
\usepackage{euler}
\usepackage{mathtools}

\usepackage{xcolor}

\usepackage{graphicx} 
\usepackage{epstopdf,epsfig}

\usepackage[hidelinks]{hyperref}
\usepackage{comment} 
\usepackage{amsmath,amsthm,amsfonts,amssymb,latexsym,amscd,enumerate,url}
\usepackage{amssymb}
\usepackage{url}
\usepackage{graphicx}
\usepackage[all]{xy}
\xyoption{arc}

\usepackage{comment} 

\newif\ifshow 
\showfalse 

\ifshow
  \includecomment{wrap}
\else
  \excludecomment{wrap} 
\fi

\numberwithin{equation}{subsection}

\swapnumbers

\theoremstyle{plain}
\newtheorem*{theorem*}{Theorem}

\newtheorem{theorem}[equation]{Theorem}
 \newtheorem{corollary}[equation]{Corollary}
\newtheorem{conjecture}[equation]{Conjecture}
 
\newtheorem{lemma}[equation]{Lemma}
\newtheorem{proposition}[equation]{Proposition}

\theoremstyle{definition}
\newtheorem{defn}[equation]{Definition}
\theoremstyle{remark}
\newtheorem{remark}[equation]{Remark}
\newtheorem*{remark*}{Remark}

\newcommand{\kappaconstant}{\gamma}

\def\R{\mathbb R}
\def\C{\mathbb C}

\def\Z{\mathbb Z}

\def\R{\mathbb R}
\def\C{\mathbb C}

\def\Z{\mathbb Z}

\def\BK{\boldsymbol{K}}
\def\BG{\boldsymbol{G}}
\def\g{\boldsymbol{\mathfrak{g}}}
\def\k{\boldsymbol{\mathfrak{k}}}

\def\s{\boldsymbol{\mathfrak{s}}}

\def\O{\mathcal{O}}

\def\F{\mathcal{F}}
\def\U{\mathcal{U}}
\def\H{\mathcal{H}}

\def\Wr{\operatorname{Wr}}
\def\wr{\operatorname{wr}}

\def\Cent{\operatorname{Centralizer}}

\def\Spec{\operatorname{Spec}}
\def\domain{\operatorname{dom}}
\def\fin{\mathrm{fin}}

\def\an{\mathit{hol}}

\def\Sol{\operatorname{Sol}}
\def\RegSol{\operatorname{RegSol}}
\def\SingSol{\operatorname{SingSol}}
\def\PhysSol{\operatorname{PhysSol}}
\def\RegSolFam{\mathcal{R}\mathit{eg}\mathcal{S}\mathit{ol}}
\def\SingSolFam{\mathcal{S}\mathit{ing}\mathcal{S}\mathit{ol}}
\def\KodSolFam{\mathcal{K}\mathit{odaira}\mathcal{S}\mathit{ol}}

\def\eigv{\lambda}
\def\action{\alpha}
\def\Sym{\mathit{Sym}}

\def\pr{\partial_r}

\begin{document}

\title[Families of Symmetries and the Hydrogen Atom]{Families of Symmetries and the Hydrogen Atom}

 \author{Nigel Higson}
 \address{Department of Mathematics, Penn State University, University Park, PA 16802, USA.}
\email{higson@psu.edu}
 \author{Eyal Subag}
 \address{Department of Mathematics, Bar Ilan University, Ramat-Gan, 5290002, Israel.}
\email{eyal.subag@biu.ac.il}

 \begin{abstract}
\noindent  
We study a new type of symmetry  for the hydrogen atom involving algebraic families of groups parametrized by the energy value in the time-independent Schr\"odinger equation. We construct an algebraic family of Harish-Chandra modules from the solutions of the  Schr\"odinger equation, and we characterize this family. We show that the subspaces of physical states may  be obtained from our algebraic family using a Jantzen filtration, and we relate our algebraic  methods with   spectral theory and  scattering theory  using the limiting absorption principle.
\end{abstract} 
\maketitle

\section{Introduction} 
It has long been known that, besides its \emph{visible} rotational symmetries, the  Schr\"{o}\-dinger equation for the hydrogen atom  possesses  \textit{hidden symmetries} that  explain the degeneracy of  energy eigen\-spaces, and help determine  the eigenvalues and eigenfunctions \cite{Pauli, Fock:1935vv,Bargmann}. These larger  symmetry groups vary from energy to energy, although the groups attached to all positive energies, and separately those attached to all negative energies,  may be---and usually are---identified with one another.

In this paper we shall develop a different approach. We shall organize the various symmetry groups into  an \emph{algebraic family}, and analyze the consequences.  We shall determine the physical states of the hydrogen atom algebraically, using   the Jantzen filtration.  We shall explain how our algebraic account is parallel to early analytic work in scattering theory by  Heisenberg \cite{Heisenberg43a,Heisenberg43b,Heisenberg44} and Kodaira \cite{Kodaira49}.  In fact we shall directly connect the limiting absorption principle in scattering theory to algebraic families.   Along the way we shall prove a classification result for representations of the kinds of algebraic families introduced here,  give a complete analysis of the Hilbert space spectral decomposition of the Schr\"odinger operator from the point of view of representation theory, and prove a variety of other results.

In physics, one seeks solutions of the time-independent Schr\"o\-dinger equation for the hydrogen atom,
$H \psi = E \psi $, 
with definite total angular momentum. One further imposes the condition that $\psi$ be square-integrable, or almost square integrable.  The linear span of all such $\psi$, as the angular momentum varies but the energy $E$ is fixed, is   what we shall call in this paper the \emph{physical solution space}, and denote by $\PhysSol(E)$.  From the physics  point of view, the spectrum of the Schr\"odinger operator is the set of all real $E$ for which the physical solution space is nonzero.  In simplified units it has the form
\begin{equation*}
	\operatorname{Spec}(H)= \Bigl \{\, - \frac{1}{\,\,n^2}   : n=1,2,\ldots \,  \Bigr\}\, \sqcup \, [0,\infty) .
\end{equation*}
 This is one of the fundamental formulas in non-relativistic quantum mechanics. 

The orthogonal group  $K{=}O(3)$ acts on each of the spaces $\PhysSol(E)$, because the Schr\"o\-dinger operator $H$ is $O(3)$-invariant.  But $\PhysSol(E)$ is also acted on by a   $6$-dimen\-sional Lie algebra
\[ 
\mathfrak{g}_E   \cong
\begin{cases}
\mathfrak{so}(3,1) & E > 0\\
 \mathfrak{so}(3)\ltimes \R^3 & E =0\\
\mathfrak{so}(4) & E < 0.
\end{cases} 
\]
This was  first discovered by  Pauli \cite{Pauli} in the case   $E<0$, and then   developed further by  Fock  \cite{Fock:1935vv},   Bargmann  \cite{Bargmann}  and many others.

Instead of making  the identifications in the display above, we shall view the Lie algebras $\mathfrak{g}_E$ as (the real forms of) some of the fibers of  an algebraic family of Lie algebras defined over the complex affine line.  The space of algebraic sections of this family is a finite-dimensional Lie algebra $\g$ over the algebra $\O$ of complex polynomial 
functions on the line, and it is   concretely  realized as  differential operators on $\R^3$.  Putting it together  with the constant family $\BK$ whose  fiber  is the visible symmetry group $O(3)$, we obtain an algebraic family $(\g,\BK)$ of Harish-Chandra pairs \cite{Ber2016}.

We shall study  the solutions of the Schr\"odinger equation using  representation theory for the family $(\g,\BK)$. Basic  ODE theory provides  \emph{regular} solutions that  extend to entire functions of the radial coordinate. 
It is a simple matter to realize the spaces $\RegSol(E)$ of regular solutions as the fibers of an $\O$-module $\RegSolFam$. The visible symmetry group $K$ acts on $\RegSolFam$, with each isotypical part being a free and finitely generated $\O$-module.  Our launching point is the following result, which is by no means evident:

\begin{theorem*}[See Section~\ref{sec-solutions}]
The module $\RegSolFam$ carries the structure of an algebraic family of Harish-Chandra modules \textup{(}as in \cite{Ber2016}\textup{)} for   $(\g,\BK)$.
 \end{theorem*}

We shall also define a family of spaces of \emph{singular solutions} (meaning only \emph{not regular}, rather than actually singular in any particular way) by taking  the quotient  of  {all} solutions by the regular solutions:
\[
\SingSol (E) = \Sol (E) / \RegSol (E)\qquad (E\in \C).
\]
These spaces, too, are the fibers of an algebraic family of Harish-Chandra modules, which we shall denote by  $\SingSolFam$.  

\begin{theorem*}
[See Section~\ref{sec-solutions}]
The Wronskian bilinear form \textup{(}from elementary differential equations\textup{)} defines a twisted-equivariant and nondegenerate pairing 
\[
\Wr \colon \SingSolFam \times \RegSolFam \longrightarrow \O.
\]
 \end{theorem*}
 
It follows that the Wronskian induces an isomorphism  from  the $\theta$-twisted dual of the family $\RegSolFam$ to the family $\SingSolFam $ (the type of twisting that is involved here is explained in Section~\ref{subsec-twisted-dual-families}). We shall also prove the following result, using representation theory techniques and without reference to the Wronskian or differential equations:

\begin{theorem*}
[See Section~\ref{sec-phys-sol-jantzen}]
There is an essentially unique, generically invertible, algebraic intertwining operator from the family $\SingSolFam$ to the family $\RegSolFam$.
\end{theorem*}

Now the space of all regular    solutions of the Schr\"odinger equation is infinite-dimensional for \emph{every} energy $E$, real or complex. It  is often much larger than the space of all physical solutions, since for some values of $E$ there are no physical solutions at all, while for others there is only a finite-dimensional space of physical solutions. One can therefore ask how the physical solutions can be extracted from the  family $\RegSolFam$.  

Our answer is that  the Jantzen filtration technique from representation theory, used as in \cite{Ber2017},  not only isolates the physical solutions among the regular solutions, but it also determines unitary inner products on the physical solution spaces.  The technique applies whenever one is given a generically invertible, algebraic family of intertwining operators acting between a family and its dual, as is the case here thanks to the previous two theorems.

\begin{theorem*}[See Section~\ref{sec-phys-sol-jantzen}]
The physical spectrum of hydrogen atom coincides with the set of all energies $E {\in} \C$ for which the fiber  $\RegSol(E)$  has  a nonzero infinitesimally   unitary Jantzen quotient. This quotient is unique, and it identifies with $\PhysSol(E )$   as a  unitary $(\g|_E ,K)$-module. 
\end{theorem*}
 
This theorem, which spotlights the role that  the algebraic family as a whole plays in determining the spectrum, calls to mind similar  features of scattering theory, and it is to these analytic issues that we turn next.

The Schr\"odinger operator for the hydrogen atom has a natural self-adjoint extension on the Hilbert space of square-integrable functions on $\R^3$. We shall give a complete analysis of  the  spectral decomposition of the Schr\"odinger operator from the perspective of unitary representation theory, and show that the spectral subspaces are the irreducible unitary representations obtained by completing the physical solution spaces. We shall also compute the spectral measure. See Section 9.  But here we shall focus on a single issue that relates the Jantzen technique to scattering theory.

Fix a value $E$ that is \emph{not} in the spectrum of the self-adjoint Schr\"o\-dinger operator $H$, and define  
\[
\Sol_0 (E),  \Sol_\infty (E)\subseteq \Sol(E)
\]
to be the spaces of all those ($K$-finite) solutions of the time-independent Schr\"o\-dinger equation  whose restrictions to neighborhoods of $0$, respectively $\infty$, coincide with the restrictions of functions in the self-adjoint domain of $H$.  As a result of the assumption that $E$ is not in the self-adjoint spectrum of $H$, these are complementary within the space  of all solutions: 
\[
\Sol (E) =   \Sol_{0} (E) \oplus \Sol_{\infty}(E).
\]
Moreover it follows from a simple analysis of the self-adjoint domain of $H$ that 
\[
\Sol(E)_0 = \RegSol(E).
\]
Now, let us say that a family  $\{ V_E\} _{E \notin \Spec (H)}$  of   eigenfunctions in the spaces $\Sol_\infty (E)$ is of \emph{Kodaira type} if its Wronskian pairing with any algebraic section of $\RegSolFam$  is   polynomial in $E$.  Such a family  determines, and is determined by, an algebraic section of the family $\SingSolFam$. 

The following is a variation on a result used by Kodaira \cite{Kodaira49} to confirm  Heisenberg's discovery \cite{Heisenberg43a,Heisenberg43b,Heisenberg44} that  the negative spectrum of the Schr\"o\-dinger operator is determined by scattering data associated to the positive spectrum. Kodaira's result is an instance of the   \emph{limiting absorption principle} in scattering theory.
 
 \begin{theorem*}[See Section~\ref{sec-resolvents}]
If  $\{ V_E\} _{E \notin \Spec (H)}$ is a Kodaira-type  family  in the spaces $\Sol_{\infty}(E)$,  then  for every $E>0$  the limits 
 \[
 \lim_{\varepsilon\to 0+}  V_{E+i\varepsilon}
 \quad \text{and} \quad 
 \lim_{\varepsilon\to 0+}  V_{E-i\varepsilon}
 \]
 both exist.  
 \end{theorem*}
 
Now,  our result relating  the Jantzen filtration to the physical spectrum of the hydrogen atom is very much in the same spirit as Heisenberg's discovery.  The following theorem makes an explicit connection between algebraic families and scattering theory. 

 \begin{theorem*}[See Section~\ref{sec-resolvents}]
The difference of the limits in the theorem above is a member of  $\RegSol (E)$. After adjusting by a Wronskian    factor, the difference varies algebraically with $E$, and the  morphism
\[
\mathcal{A}\colon \SingSolFam \longrightarrow \RegSolFam .
\]
that is determined in this way is the unique generically invertible intertwining operator between  algebraic families of Harish-Chandra modules for $(\g,\BK)$ \textup{(}from which the physical states of the hydrogen atom may be obtained using the Jantzen technique\textup{)}.
\end{theorem*}

We believe that this unexpected bridge back to algebra  from analysis makes it apparent that the approach to hidden symmetries  via algebraic families, particularly algebraic families of solutions, merits close attention.

\subsection*{Acknowledgments}   The late Joseph Birman  drew the authors' attention to the hydrogen atom system, and suggested that scattering phenomena there should be investigated further. The authors  are  also     grateful to  Joseph Bernstein for many useful remarks.

\section{The family of Harish-Chandra pairs  for the hydrogen atom}
\label{sec-algebraic-family}
In this section we shall associate      to the  Schr\"{o}dinger operator \eqref{eq-schrodinger-op}  an   algebraic family of Harish-Chandra pairs that organizes the various hidden symmetries  of the Schr\"odinger equation.

\subsection{Rescaled Schr\"odinger operator}
The time-inde\-pendent   Schr\"{o}dinger equation for the hydrogen atom  is  the eigenvalue equation
\begin{equation}
\label{eq-schrodinger-eqn}
 H\psi =E \psi ,
 \end{equation}
 where  $\psi$ is a function on $\R^3$, $E$ is a real number (the energy), and $H$ is the Schr\"od\-inger operator 
\begin{equation}
\label{eq-schrodinger-op}
H=-\frac{\hbar^2}{2\mu}\triangle-\frac{e^2}{r}.
\end{equation}
Here  $\hbar$ is the reduced Planck's constant, $\mu$ is the reduced mass, $e$ is the electron charge and $r$ is the distance in $\R^3$ from the origin. 
To streamline formulas  we shall work not with the operator $H$ in \eqref{eq-schrodinger-op} but with the   \emph{rescaled Schr\"odinger  operator} 
\begin{equation}
\label{eq-T-rescales-H}
T = \frac{2\mu}{\,\hbar^2} H .
\end{equation}
It follows from the definition above that
$
T = - \Delta -  {2 \kappaconstant }/{r}$ ,
where 
\begin{equation}
\label{eq-kappa}
\kappaconstant  = \frac{ \mu e^2}{\hbar^2} .
\end{equation}
We shall use $T$ and $\kappaconstant $ throughout the paper (we might have chosen units so that $\kappaconstant  {=} 1$, but we resisted the temptation to do so).   

To familiarize the reader   with the new notation, we note that the computation of the spectrum of $H$ mentioned in the introduction is equivalent to the formula
\[
	\operatorname{Spec}(T)= \Bigl \{\, - \frac{\kappaconstant ^2}{ n^2}   : n=1,2,\ldots \,  \Bigr\}\, \sqcup \, [0,\infty) .
\]

\subsection{The infinitesimal hidden symmetries}\label{cent} 
In this subsection we shall consider $T$ as an element in 
the  algebra of  linear partial differential operators on  
the space  $\R_0^3 = \R^3 \setminus \{ 0\}$ 
 with smooth function coefficients.  We shall denote by $\Cent(T)$ the centralizer of the rescaled Schr\"odinger operator in this algebra.   

Since $T$ is rotation-invariant, the most evident elements in $\Cent(T)$ are the  infinitesimal rotation operators 
\begin{equation}
\label{eq-l-ops}  
 L_1 =x_2\partial_3-x_3 \partial_2 , \quad   
 L_2 =x_3\partial_1-x_1\partial_3 \quad \text{and} \quad  
 L_3=x_1\partial_2-x_2\partial_1  .
 \end{equation}
But in addition, the centralizer includes  the so-called Runge-Lenz operators
\begin{equation}
\label{eq-r-ops}
\begin{aligned}
R_1 &  =\sqrt{-1} \Bigl(L_3\partial_{2}-L_2\partial_{3}-\partial_1+  \frac{\kappaconstant  x_1}{ r} \Bigr), 
\\ 
 R_2 &  =\sqrt{-1} \Bigl(L_1\partial_{3}-L_3 \partial_{1} -\partial_2+ \frac{\kappaconstant  x_2}{ r} \Bigr),
\\
 R_3&  =\sqrt{-1} \Bigl(L_2\partial_{1}-L_1\partial_{2}-\partial_3+  \frac{\kappaconstant  x_3}{r}\Bigr),
 \end{aligned}
 \end{equation}
 which Pauli obtained by quantizing the similarly-named    quantities in the classical Kepler problem (see for example \cite[Chs.\ 1\&2]{GiulleminSternberg90} or \cite[Ch.\ 18]{Hall2013}).

\begin{defn}
\label{def-g}
We shall denote by $\O$ the subalgebra  of $\Cent(T)$  consisting of polynomials in $T$.  It is an embedded copy of the  algebra of polynomial functions on the line.  In addition, we shall denote by $\g$ the $\O $-linear span of the operators 
\[
\{L_1,L_2,L_3\}\quad \text{and} \quad \{R_1,R_2,R_3\}
\]
within $\Cent(T)$.
\end{defn}

 \begin{lemma}
\label{lem-Lie-structure-on-g}
The space $\g$ is   a free $\O $-module on the elements in Definition~\ref{def-g}, and a  Lie algebra over $\O $ under the commutator bracket.
 \end{lemma}

 \begin{proof}
 By a direct calculation,
\begin{equation}
 \label{eq-lr-relations}
 \begin{aligned}
[L_i,L_j] & =-\varepsilon_{ijk} L_k \\ 
[R_i,R_j]& =\phantom{-}\varepsilon_{ijk} T L_k  \\ [L_i,R_j]& =-\varepsilon_{ijk}R_k ,
\end{aligned}
\end{equation}
 where $\varepsilon_{ijk} $ is  the Levi-Civita symbol (which is $0$ if an index is repeated, and the sign of the permutation $(i,j,k)$ otherwise). This shows that $\g$ is a Lie algebra.  Freeness can be checked by examining leading-order terms (that is, the principal symbols) of the operators $T^kL_i $ and $T^kR_i$.  The former have order exactly $2k{+}1$, and for each $k$ are linearly independent over $\C$; the latter have order exactly $2k{+}2$, and again for each $k$ they are linearly independent over $\C$.  The lemma follows from this.\end{proof}

Because  $\g$  is finitely generated and free as a module over $\O $, and also a Lie algebra over $\O $, it may be viewed as (the Lie algebra of global sections of)   an algebraic family of complex  Lie algebras over  $\C$, the maximal ideal spectrum of $\O$.    The individual fibers of the family are the complex Lie algebras 
\[
\g\vert _ \lambda = \g \,/\, I_\lambda \cdot \g,
\]
where $\lambda \in \C$ and where $I_\lambda \subseteq  \O $ is the ideal generated by $T-\lambda$ (in the text we shall denote typical points on the line by $\lambda$ rather than by $E$ as we did in the introduction).   

To better understand these fibers,  let us give a concrete realization of the family $\g$ as an algebraic family of matrix Lie algebras. Define a morphism from the $\O $-module $\g$ into the $4{\times}4$-matrices over $\O $ by
\begin{equation}
\label{eq-R-and-L-matrices}
\begin{gathered}
L_1 \mapsto 
\left[\begin{smallmatrix}
\,0\, & \,0\, & \,0\, & \,0\,\\
0  &0 & {-}1 &  0 \\
 0 &  1 & 0 & 0 \\
  0 &  0 & 0  &0
\end{smallmatrix}\right] ,
\quad 
L_2\mapsto 
\left[\begin{smallmatrix}
\,0\, & \,0\, & \,-1\, & \,0\,\\
 0 &0 & 0 &  0 \\
 1 &  0 & 0 & 0 \\
  0 &  0 & 0  &0
\end{smallmatrix}\right]
\quad\text{and} \quad 
L_3 \mapsto 
\left[\begin{smallmatrix}
\,0\, & -1 & \,0\, & \,0\,\\
1 &0 & 0 &  0 \\
 0 &  0  & 0 & 0 \\
  0 &  0 & 0  &0
\end{smallmatrix}\right]
\\
R_1 \mapsto 
\left[\begin{smallmatrix}
\,0\, & \,0\, & \,0\, & \phantom{.}1\phantom{.}\\
0 &0 & 0 &  \,0\, \\
 0 &  0 & 0 & 0 \\
  \phantom{.}T\phantom{.} &  0 & 0  &0
\end{smallmatrix}\right] ,
\quad 
R_2\mapsto 
\left[\begin{smallmatrix}
\,0\, & \,0\, & \,0\, & \,0\,\\
 0 &0 & 0 &  -1 \\
 0 &  0 & 0 & 0 \\
  0 &  -T & 0  &0
\end{smallmatrix}\right]
\quad\text{and} \quad 
R_3 \mapsto 
\left[\begin{smallmatrix}
\,0\, & \,0\, & \,0\, & \,0\,\\
0 &0 & 0 &  0 \\
 0 &  0 & 0 & \phantom{.}1\phantom{.} \\
  0 &  0 & T &0
\end{smallmatrix}\right] .
\end{gathered}
\end{equation}
The image matrices are linearly indendent over $\O $, so the morphism is injective. The explicit computations in the proof of Lemma~\ref{lem-Lie-structure-on-g} show that it is also a Lie algebra homomorphism.

The image of $\g$  is the Lie algebra of all matrices over $\O $ of the form
\[ 
\begin{bmatrix}
0 & e(T) & f(T) & p(T)\\
-e(T)&0 & g(T)& q(T)\\
 -f(T)&  - g(T)& 0 &r(T)\\
 Tp(T)& Tq(T)& Tr(T) &0
\end{bmatrix},
\]
where $e$, $f$, $g$, $p$, $q$ and  $r$ are polynomial functions.  The fiber Lie algebra $\g\vert _\lambda  $ is therefore isomorphic to  Lie algebra of all matrices of the same form,  but with $T$ replaced by $ \lambda \in \C $. 

When $\lambda {=}0$, the fiber  is the semidirect product $\mathfrak{so}(3,\C)\ltimes \C^3$.
When $\lambda  {\ne} 0$,  it  is the complex Lie algebra of infinitesimal symmetries of the symmetric bilinear form 
\[
S_ \lambda = \operatorname{diag}(1,1,1, -  \lambda).
\]  
If we fix a complex number $\mu$ such that 
\begin{equation}
\label{eq-def-mu}
\mu^2 \lambda = -1
\end{equation}
 and if we define 
\begin{equation}
\label{eq-def-lambda}
L_j^\pm = \tfrac 12 \left ( L_j \pm \mu R_j\right ) ,
\end{equation}
then  
\[
[ L _i^\pm, L _j^\pm] = -\varepsilon _{ijk} L ^\pm _k ,
\]
and moreover 
\[
[ L_i^+ , L _j ^-] = 0 .
\]
So  the elements $L^\pm _j$  determine an isomorphism of complex Lie algebras 
\begin{equation}
\label{eq-splitting-gz}
\g\vert_\lambda  \cong \mathfrak{so}(3) \times \mathfrak{so}(3) 
\end{equation}
when $\lambda {\ne} 0$. The isomorphism depends on the choice of $\mu$; the two possible choices lead to isomorphisms that differ from one another by the flip automorphism of the product Lie algebra.  Note also that under the isomorphism, the subalgebra of $\g\vert _ \lambda $ spanned by the elements $L_j$ is mapped to the diagonal subalgebra of the product in \eqref{eq-splitting-gz}.

From the embedding \eqref{eq-R-and-L-matrices} we see that  $\g$ is the family of Lie algebras associated to an algebraic family $\BG$ of   groups (that is, a smooth group scheme over the complex line; see \cite[Sec.\ 2.2]{Ber2016}), namely  
\begin{equation}
\label{eq-BG-def}
\BG = \bigl \{ \, (W, \lambda) \in GL(4,\C){\times} \C :  \det(W) = 1 \,\, \text{and} \,\, W  S_ \lambda W^t = S_ \lambda\bigr \}  .
\end{equation}
The fiber groups are isomorphic to 
\[
 \begin{cases}  
 O(3,\C)\ltimes \C^3 & \lambda =0 \\
SO(4,\C) &  \lambda\ne 0  .
 \end{cases}
 \]  
Here the action of $O(3,\C)$ on $\C^3$ in the semidirect product is given by 
\begin{equation}
\label{eq-O3-action}
A\colon v \longmapsto \det(A)^{-1} Av .
\end{equation}
This is, of course, not quite the standard matrix action  on $\C^3$. But it  and the corresponding action of the compact group $O(3)$ on $\R^3$ will reappear throughout the paper.

\subsection{The visible symmetries and Harish-Chandra pairs }
The constant family of groups 
\[
\BK = O(3,\C) {\times} \C
\]
may be viewed as a subfamily $\BG$  via the embedding 
\[
A \longmapsto \begin{bmatrix} A & 0 \\ 0 & \det(A)^{-1} \end{bmatrix} .
\]
The family $\BK$ acts on the family $\g$  via the isomorphism \eqref{eq-R-and-L-matrices} and the adjoint action. This   makes
$(\g,\BK)$ into an algebraic family of Harish-Chandra pairs in the sense of \cite{Ber2016,Ber2017}.

Note that in terms of differential operators, the action of $\BK$ on $\g$ is induced from the action of  $O(3)$  on    $\R^3_0 $ via the formula \eqref{eq-O3-action}.

\subsection{Real structure}
\label{subsec-real-structure}

The family  $(\g,\BK)$ carries a natural \emph{real structure}, which is to say a  set of compatible conjugate-linear involutions on $\O$,    $\g$ and  $\BK$. See \cite[Sec.\ 2.5]{Ber2016} for details concerning the definition, but the formulas below should make those details clear.

First, we define a   complex-conjugate-linear involution 
\[
\sigma\colon \O\longrightarrow \O
\]
by means of the standard formula $\sigma(p)(z) = \overline{p(\overline{z})}$. Here we are viewing elements of $\O$ as polynomial functions on the line.  Thinking of them as polynomials in $T$, the involution acts by complex-conjugating all the coefficients of the polynomial.

Next, the elements  of $\g$, being linear partial differential operators on   $\R^3_0$, have formal adjoints.   The family $\g$ is stable under the formal adjoint  operation since $T$ is formally self-adjoint, while  each of the  operators $L_i$ and $R_i$  in \eqref{eq-l-ops} and \eqref{eq-r-ops} is formally skew-adjoint. Using this fact, we define  
\begin{equation}
\label{eq-real-structure-on-g}
\sigma_{\g} \colon \g\longrightarrow \g  
\end{equation}
by $ \sigma_{\g}(X)=-X^*$. 
 This preserves the Lie bracket and satisfies 
\begin{equation}
\label{eq-compatible-involutions}
\sigma_{\g} (p\cdot X) = \sigma(p)\cdot  \sigma_{\g} (X)
\end{equation}
for all $p \in \O $ and all $X \in \g$. 

Finally we define an involution  on   $\boldsymbol{K}$ by combining the involution on the base with the antiholomorphic involution 
 \[
 k \longmapsto (k^*)^{-1}
  \]
on the fiber $O(3,\C)$.
 The three  involutions so-defined constitute a real structure on the algebraic family of Harish-Chandra pairs  $(\g,\BK)$.

 \subsection{Cartan involution}
\label{subsec-cartan}
There is a second natural complex-conjugate-linear involution
\begin{equation}
\label{eq-complex-conj-inv}
\sigma'_{\g}\colon \g \longrightarrow \g ,
\end{equation}
which is defined as follows: each element of $\g$ is   a linear differential operator on $\R^3_0$, and $\sigma'_{\g}$ replaces each coefficient function in the operator by its complex conjugate.  Like the involution $\sigma_{\g}$, it is compatible with the involution on $\O$ in the sense of the formula \eqref{eq-compatible-involutions}. Moreover $\sigma_{\g}$ and $\sigma'_{\g}$ commute with one another.  So the composition 
\begin{equation}
\label{eq-cartan-inv}
\theta = \sigma_{\g}\circ \sigma'_{\g}\colon \g \longrightarrow \g
\end{equation}
is an $\O$-linear involution of $\g$.  Its fixed subalgebra is precisely the algebra $\k$ of visible symmetries, while its $-1$ eigenspace is spanned by the Runge-Lenz operators.

\begin{remark}
The   symmetries of the Schr\"odinger equation are related to our involutions as follows: 
the visible symmetries are those elements of $\g$ that are fixed by both $\sigma_{\g}$ and $\theta$, whereas the hidden symmetries are those fixed by $\sigma_{\g}$ alone.
\end{remark}

\subsection{Involutions at the group level}

 The involutions $\theta $ and $\sigma_{\g}$ assume a simple form under  the isomorphism in \eqref{eq-R-and-L-matrices} from $\g$ to a family of $4{\times}4$-matrix Lie algebras.  On the matrix fibers, $\sigma_{\g}$ is simply entry-wise complex conjugation, whereas $\theta$  is conjugation by the matrix 
 \[
 \Theta = \operatorname{diag}(1,1,1,-1).
 \]
   So these involutions lift to the family of groups $\BG$. To summarize:

  \begin{theorem}\label{ThmFamilyOfGroups}
There is an algebraic family of complex algebraic groups $\BG$ over the complex line, together  with a real structure $\sigma$ and a commuting involution $\theta$  such that

\begin{enumerate}[\rm (a)]

\item The associated algebraic family of Lie algebras is the   family $\g$. 

\item The family $\BG$ includes the constant family $\BK$ with fiber $O(3,\C)$ as a subfamily, and the real structure restricts to the standard real structure on this subfamily given by entry-wise complex conjugation.

\item The induced real structure on $\g$ is the involution  in \eqref{eq-real-structure-on-g}, and the fixed groups for the real structure over the points $\lambda\in \R$ are the  hidden symmetry  groups
 \[
\BG|_\lambda ^{\sigma}\cong \begin{cases}
SO(3,1) & \lambda >0 \\
O(3)\ltimes \R^3 & \lambda =0 \\
SO(4) & \lambda < 0.
\end{cases}
\] 

\item The fixed family for the     involution $\theta$  is $\BK$, and the induced action at the Lie algebra level is the involution in \eqref{eq-cartan-inv}. 
\qed

\end{enumerate}

\end{theorem}

\begin{remark}
As explained in \cite{Bar2017}, the whole family $\BG$ can actually be constructed from the  commuting involutions $\sigma$ and $\theta$ on the fiber over $\lambda {=}1$.
\end{remark}

\begin{remark}
Starting from  Fock's work \cite{Fock:1935vv}, a substantial literature has developed from  the fact that the symmetry groups described in the theorem above are isomorphic to one another for all  negative   $\lambda$, and also for all positive $\lambda$.  One remarkable development, with ramifications that continue to be explored, has been the realization of  \emph{all} the physical solutions of the time-dependent Schr\"odinger equation associated to all $\lambda$ inside a \emph{single} irreducible unitary representation of a larger group, namely the minimal representation of  (the connected component of the identity of) $O(4,2)$, called the \emph{dynamical symmetry group}.  See \cite{KobayashiOrstedI,KobayashiOrstedII,KobayashiOrstedIII} for a comprehensive mathematical treatment of the minimal representation in this case, and  see \cite{Maclay2020} or \cite{MengZhang2011,Meng2010} for recent accounts from the physics perspective.  

The approach to the symmetries of the hydrogen atom via the minimal representation of $O(4,2)$ is rather different   from our aims here,  which are to view the Schr\"odinger operator as a sort of Casimir operator for an algebraic family of groups,  and then understand how  the eigenspaces  vary algebraically as individual irreducible representations of the groups in this family.
 But it is interesting to consider possible connections between the two.  
 
To list one such possibility, away from $\lambda {=}0$, the isomorphisms in Theorem~\ref{ThmFamilyOfGroups} lead to an embedding of the family of groups studied in this paper into the constant family with fiber (the complexification of) $O(4,2)$.  The embedding becomes algebraic if one extends $\O$ by adjoining a square root of $T$.  (Another way to understand the embedding is to note that our family, away from $\lambda =0$,  is naturally a subfamily of the family of orthogonal groups preserving the forms 
\[
Q_\lambda = 
\operatorname{diag} \bigl ( 1,1,1,\lambda, -1,-\lambda^{-1} \bigr ) 
 \qquad (\lambda \ne 0) ,
\]
while conjugation by
\[
A_\lambda = 
\operatorname{diag} \bigl ( 1,1,1,\lambda^{1/2}, 1,\lambda^{-1/2} \bigr ) 
\]
transforms these orthogonal groups into the standard group $O(4,2)$.)   The embedding  of groups is compatible with the embedding of the solutions of the Schr\" od\-inger equation into the minimal representation. It therefore becomes  an interesting problem to \emph{construct} the algebraic families of solutions analyzed in this paper by starting  from the minimal representation. 
We hope to take up this issue elswhere, and are obliged to the referee for suggesting this type of investigation.
\end{remark}

\subsection{The enveloping algebra and its center} 
Denote by $\mathcal{U}(\g)$ the enveloping algebra of the $\O $-Lie algebra $\g$.  See \cite[Section 2.7]{Bourbaki60} for general information on enveloping algebras applicable to Lie algebras over commutative rings, including for instance the Poincar\'e-Birkhoff-Witt theorem, which applies to our case because $\g$ is free as a module over $\O $.

We shall be especially concerned with the center $\mathcal{Z}(\g)$ of the enveloping algebra, and we shall give here an analysis of its structure.
It follows   from the formulas \eqref{eq-lr-relations} that the   elements
\begin{equation}
\label{eq-two-casimirs}
  RL  
\quad \text{and} \quad
   T L^2 - R^2  
 \end{equation}
 belong to the center of the enveloping algebra (we are using standard abbreviated notation, so that $R^2 = R_1^2 + R_2^2 + R_3^2$, and so on).  Moreover these elements  are $O(3)$-invariant.  

\begin{proposition} 
The center of the enveloping algebra $\mathcal{U}(\g)$ is freely generated, as a commutative $\O $-algebra, by   $ RL  $
and 
 $ T L^2  - R^2$. 
 \end{proposition}

\begin{proof}
View $\mathcal{U}(\g)$ as a  complex associative algebra, and equip it with the increasing filtration for which 
\begin{equation}
\label{eq-tech-filtration}
\operatorname{order}(L_i) = 0, \quad \operatorname{order}(R_i) = 1 \quad \text{and} \quad \operatorname{order}(T) = 0.
\end{equation}
The associated graded algebra is  the tensor product 
$
\mathcal{U}(\mathfrak{g}_0) \otimes_{\C} \O ,
$
where $ \mathfrak{g}_0$ is the $6$-dimensional graded complex Lie algebra with basis elements 
$\ell_1$, $\ell_2$, $\ell_3$ in degree 0 and $r_1$, $r_2$, $r_3$ in degree 1, and relations
\begin{equation}
\label{eq-tech-associated-graded}
[\ell_i,\ell_j]  =-\varepsilon_{ijk} \ell_k , \quad 
[r_i,r_j]  = 0  \quad \text{and}\quad [\ell_i,r_j] =-\varepsilon_{ijk}r_k ;
\end{equation} 
compare to formulas \eqref{eq-lr-relations} earlier.    

Now, the center of the associated graded algebra is
$
\mathcal{Z}(\mathfrak{g}_0) \otimes_{\C} \O ,
$
and by an explicit computation   $\mathcal{Z}(\mathfrak{g}_0)$ is a free commutative algebra over $\C$ on the elements $  r\ell $ and $ r^2$.  Compare  \cite[Theorem 2.1]{Gonzalez88}. 
The proposition follows  from this by a standard induction argument on the order of elements in the center of the enveloping algebra.
\end{proof}

\begin{remark} 
The reference \cite[Theorem 2.1]{Gonzalez88}  actually computes the $O(3)$-invariant part of $\mathcal{Z}(\mathfrak{g}_0)$ for the \emph{standard} action of $O(3)$ on $\C^3$, not the action \eqref{eq-O3-action}. For that action the invariant part is freely generated by $(r\ell)^2$ and $r^2$. The element $r \ell$ is not $O(3)$-invariant for the standard action; instead, it transforms under   $A{\in} O(3)$ as 
\[
  r\ell  \longmapsto \det (A)\cdot r \ell .
\]
But the computation of  $\mathcal{Z}(\mathfrak{g}_0)$  is easily reduced to the computation of the standard $O(3)$-invariant part using the formula 
\[
\mathcal{Z}(\mathfrak{g}_0) = \mathcal{Z}(\mathfrak{g}_0)^{O(3),\mathrm{standard}} \oplus  r \ell \cdot \mathcal{Z}(\mathfrak{g}_0)^{O(3),\mathrm{standard}} .
\]
\end{remark}

\subsection{Realization in differential operators} 
Because $\g$ is defined as a Lie algebra of linear differential operators   commuting with the rescaled Schr\"od\-inger operator $T$, there is a canonical morphism of $\O $-algebras
\begin{equation}
\label{eq-morphism-into-centralizer}
\action \colon \mathcal{U} (\g) \longrightarrow \Cent(T) .
\end{equation}
We shall briefly examine both the image and the kernel of this morphism.

 The algebra $\Cent (T)$ is filtered by the usual order of differential operators. The following lemma is proved by a direct computation, which we shall omit:

 \begin{lemma}
\label{lem1}
The order $2$ part of $\Cent(T)$  is  spanned by the totality of the elements 
\begin{enumerate}[\rm (i)]
\item 
${I}$  \textup{(}the identity operator\textup{)}
\item
$T$ \textup{(}the rescaled Schr\"{o}dinger operator\textup{)}  
\item $L_1$, \, $L_2$ \, and \, $L_3 $
\item $R_1$, \, $R_2$  \, and \, $R_3 $
\item $L_1^2+L_2^2+L_3^2$
\item  $L_1L_2+L_2L_1$, \, $L_1L_3+L_3L_1$, \,$L_2L_3+L_3L_2$, $L_1^2-L_2^2$ \, and \, $L_2^2-L_3^2 $. \qed
\end{enumerate}
\end{lemma}
We grouped the elements into six separate sets because  the operators in each span an irreducible representation of the group $O(3)$.

The following formulas for the action of $\mathcal{Z}(\g)$ under the morphism $\action$ will play important roles in the sequel.  For the proofs, see for example \cite[Prop.\ 18.11]{Hall2013}.

\begin{lemma} 
\label{lem-vals-of-casimirs}
The images of the central elements $RL$ and $ T L^2 -R^2  $ in \eqref{eq-two-casimirs} under the homomorphism $\action$ are given by the formulas 
\[
\action (RL  ) = 0 
\]
and
\[
\pushQED{\qed}
\action (T L^2 -R^2 ) =    T +   {\kappaconstant ^2} I .
\qedhere
\popQED
\]
\end{lemma}

\begin{remark}
It is evident from Lemma~\ref{lem1}  that the range of \eqref{eq-morphism-into-centralizer} includes the full second-order part of the centralizer.  Indeed it is asserted  in \cite{DrokinEtAl96}  that the full centralizer is generated by $\g$.  This is consistent with experience from the study of classical superintegrable systems, such as the Kepler system, where, in  $n$ dimensions there are at most  $2n{-}1$ functionally independent     conserved quantities; see for example \cite[p.11]{Miller13}. The analogous   statement for quantum superintegrable systems  is known in some cases and there are no known counterexamples; see \cite[p.12]{Miller13}.   
It also seems likely to us that the kernel of the morphism from $\mathcal{U}(\g)$ into the centralizer is fully described by the relations in Lemma~\ref{lem-vals-of-casimirs}. 
\end{remark} 

\begin{wrap}
 \begin{conjecture}
The centralizer of $T$ coincides with the  associative algebra generated by $\g$.  Moreover the kernel of the morphism  \eqref{eq-morphism-into-centralizer} is the two-sided ideal in the enveloping algebra generated by the relations 
\[
RL  = 0  \quad \text{and} \quad  TL^2 - R^2  = T  +   {\kappaconstant ^2} I .
\]
\end{conjecture}
\end{wrap}

\section{Families of Harish-Chandra modules}\label{Classification}

In this section, we shall collect some general information about families of Harish-Chandra modules over the family of Harish-Chandra pairs that we constructed in the previous section,  including definitions and a classification theorem.  This is in preparation for our representation-theoretic analysis of the solutions of the Schr\"odinger equation in Section~\ref{sec-solutions}.

 \subsection{Locally finite-dimensional group representations}
Let $V$ be a complex vector space that is equipped with a linear representation of the group $K=O(3,\C)$.  There is a natural evaluation map
\begin{equation}
\label{eq-locally-finite-rep}
\bigoplus_{\sigma \in \widehat K} \, W_\sigma \otimes \operatorname{Hom}_K (W_\sigma, V)\longrightarrow V ,
\end{equation}
where the sum is over the equivalence classes of irreducible (finite-dimen\-sional) linear representations of the complex algebraic group $K$, and $W_\sigma$ is a representative of the equivalence class $\sigma$.

The representation $V$ is said to be  \emph{locally finite-dimen\-sional} if the map  \eqref{eq-locally-finite-rep}  is a vector space isomorphism. 
In any case the image, which we shall denote by$V_\fin$,  is a locally finite-dimensional representation of $K$ in its own right.  The individual images $V^\sigma$  of the summands in \eqref{eq-locally-finite-rep}  are the \emph{isotypical summands} of the representation. They are the same for $V$ and $V_\fin$.

Denote by $\O(K)$ the space of regular functions on $K$. A locally finite-dimensional representation of $K$ on $V$  determines, and is determined by,  the vector space morphism
\begin{equation}
\label{eq-coaction}
V \longrightarrow\O(K)  \otimes_\C  V
\end{equation}
for which composition with evaluation at $k\in K$ (a map $\O(K)\to \C$) gives the action of $k$ on $V$. 

A dual form of \eqref{eq-coaction} may be defined as follows.  Given any locally finite-dimensional representation of $K$ on a vector space $V$, let us  denote by
\[
V^* = \operatorname{Hom}_{\C}( V, \C)_\fin
\]
the locally finite contragredient representation. 
From  the left translation action of $K$ on $\O(K)$ we obtain in this way
\[
R(K) = \O(K)^* .
\]
The morphism \eqref{eq-coaction} then corresponds to a morphism
\begin{equation}
\label{eq-hecke-action1}
R(K) \otimes _\C V \longrightarrow V .
\end{equation}
Applying this in the case where $V$ is $R(K)$ itself, we get a morphism of vector spaces
\[
R(K)\otimes_{\C} R(K) \longrightarrow R(K) .
\]
This is an associative product on $R(K)$ and \eqref{eq-hecke-action1} is a module action.\footnote{The algebra $R(K)$ is often called the \emph{Hecke algebra} of $K$. After fixing a Haar measure on $O(3)$, using which we can identify elements of $R(K)$  with functions on $O(3)$,  the product corresponds to convolution of functions on  $O(3) $. The module structure may also be described using convolution.  See \cite{KnappVogan}.} In fact a locally finite-dimensional representation of $K$ is exactly the same  as a left $R(K)$-module that is nondegenerate in the sense that $R(K)\cdot V = V$. 

Examining \eqref{eq-hecke-action1} in the case where $V$ is the direct sum of all irreducible representations of $K$, up to  equivalence, we obtain the \emph{Peter-Weyl isomorphism} of algebras 
\begin{equation}
\label{eq-peter-weyl}
R(K) \stackrel \cong \longrightarrow \bigoplus _{\sigma\in \widehat K} \operatorname{End} (W_\sigma) .
\end{equation}
One consequence of the Peter-Weyl isomorphism is  that $R(K)$ contains many central idempotents, corresponding to the identity operators on the spaces $W_\sigma$. The idempotent associated to the trivial representation will be especially important in what follows.

\begin{defn}
\label{def-triv-rep-idempotent}
We shall denote by  $e{\in} R(K)$ the central idempotent that is mapped to the identity operator on the trivial representation under \eqref{eq-peter-weyl}, and is mapped to zero in any other summand of \eqref{eq-peter-weyl}.
\end{defn}

Finally, let us establish some notation for later use.  Throughout this paper we shall be considering only representations of $K$ for which   the element $-I \in K$ acts trivially. As a result, the irreducible representations where $-I$ acts nontrivially will play no role.  The group $K$ is the direct product of its $2$-element center, generated by $-I$, and its identity component, which is of course $SO(3,\C)$.  So effectively we shall be studying representations of $SO(3,\C)$, and we shall call these the \emph{$SO(3)$-type representations} of $K$.
They may be listed as 
\[
\bigl \{ \, \sigma_\ell : \ell = 0,1,\dots \, \bigr \},
\]
 where $\sigma_\ell$ has dimension $2\ell {+}1$. We shall write $W_\ell$  in place of $W_{\sigma_\ell}$.

\subsection{Families of locally finite-dimensional representations}

By an \emph{algebraic family of representations of $K$} we shall mean a locally free, and hence in our context free, $\O$-module $\F$ that is equipped with a locally finite-dimensional representation of $K$  by $\O$-module automorphisms.   

Each of the isotypical summands of $\F$ is now a free $\O$-module, and we shall say that $\F$ is \emph{admissible} if all its isotpyical summands are finitely generated (and necessarily free) $\O$-modules.  Finally the locally finite contragredient of the family is 
\begin{equation}
\label{eq-dual-family-K-rep}
\F ^* = \operatorname{Hom}_{\O } ( \F , \O )_\fin .
\end{equation}

\subsection{Algebraic families of Harish-Chandra modules}

An   \emph{admissible algebraic family of Harish-Chandra modules} for   $(\g,\boldsymbol{K})$ is an admissible algebraic family of $K$-representations $\F$ that is equipped with an $\O$-linear action of $\g$,
\[
\g\otimes_{\O} \F \longrightarrow \F ,
\]
that is compatible with the $K$-action in the sense that the action map is $K$-equivariant, and its restriction to $\k\subseteq \g$ is the infinitesimal form of the $K$-action.

As was the case with locally finite-dimensional $K$-actions, it is possible to describe $(\g,\boldsymbol{K})$-module structures using an appropriate  associative algebra. Form the tensor product 
\[
\U(\g,\BK) = \U(\g) \otimes _{U(\mathfrak{k})} R(K) .
\]
The actions of $\g$ and $K$ on $\F$ determine a morphism
 \begin{equation}
\label{eq-hecke-action2}
\U(\g,\BK)\otimes{_\O} \F \longrightarrow \F .
\end{equation}
The $\O$-module $\U(\g,\BK)$ itself carries actions of $\g$ and $K$, the later being the tensor product of the adjoint and left-translation actions, and so we obtain a morphism
\[
\U(\g,\BK)\otimes_{\O} \U(\g,\BK)\longrightarrow \U(\g,\BK)
\]
which is an $\O$-linear associative product on $\U(\g,\BK)$.  The morphism \eqref{eq-hecke-action2} gives $\F$ the structure of a $\U(\g,\BK)$-module, and in this way algebraic families of Harish-Chandra modules for $(\g,\BK)$ correspond exactly to $\U(\g,\BK)$-modules that are free as $\O$-modules and nondegenerate in the sense that $\U(\g,\BK)\cdot \F = \F$.
For all this, see \cite{KnappVogan} (which treats the case of individual Harish-Chandra modules, but  families may be handled in the same way).

\subsection{The spherical Hecke algebra}

In this section we shall classify certain algebraic families of Harish-Chandra modules.  Later on we shall use the classification to characterize the  families that arise from the solution of the Schr\"{o}d\-inger equation (see Sections~\ref{subsec-singsolfam} and \ref{subsec-ndg-hermitian-forms}) and to prove the existence of intertwining operators between these families  (see Section~\ref{subsec-intertwining}).

\begin{defn}  
\label{defn-standard-spherical}
Let us say that an algebraic family $\F$ of Harish-Chandra modules for $(\g,\BK)$ is \emph{standard and spherical} if 
\begin{enumerate}[\rm (i)]

\item The $K$-isotypical decomposition of $\F$  consists of the full set of  $SO(3)$-type representations, each occurring with multiplicity one.

\item $\F$  is generated by its spherical vectors (that is, by  its $K$-fixed elements) in the sense that $\U(\g) \cdot \F^0 = \F$, where $\F^0$ is the isotpyical subspace for the trivial representation of $K$.

\end{enumerate}
\end{defn} 

\begin{remark}
In most situations involving families of irreducible Harish-Chandra modules, item (i) above is implied by item (ii).
\end{remark}

\begin{defn}
\label{def-value-of-casimir}
If $\F$ is standard and spherical, then the  element $TL^2 - R^2\in \mathcal{Z}(\g)$ acts on $\F^0$, and indeed on $\F$, as multiplication by some polynomial $Q\in \O$.   
We shall call  $Q$ the \emph{value} of $TL^2{ -} R^2$ on $\F$.
\end{defn}

\begin{theorem}
\label{thm-spherical-classification}
A  standard and spherical family of Harish-Chandra modules for $(\g,\BK)$ is determined up to isomorphism by the value in $\O$ of $TL^2 {-}R^2$. Moreover any element of $\O$ can arise as the value of $TL^2 {-}R^2$ in a standard and spherical family.
\end{theorem}

We shall prove this by means of the following auxilliary constructions.

\begin{defn}
The \emph{spherical Hecke algebra} for the family $(\g,\BK)$ is the following $\O$-algebra:
\[
\H(\g,\BK) = e\cdot  \U(\g,\BK)\cdot e.
\]
Here $e\in R(K)$ is the idempotent in Defintiion~\ref{def-triv-rep-idempotent}, and we view it as an element of $\H(\g,\BK)$ via the embedding $e\mapsto 1{\otimes}e$. Let us also define 
\[
\U\H(\g,\BK) = \U(\g,\BK)\cdot e
\]
This is a $\U(\g,\BK)$-$\H(\g,\BK)$-bimodule by left and right multiplication.
\end{defn}

\begin{proposition} 
\label{prop-Hecke-algebra}
The spherical Hecke algebra $\H(\g,\BK) $ is freely generated as an $\O$-algebra by the \textup{(}tensor\textup{)} product of  $TL^2 - R^2 \in \mathcal{Z}(\g)$ with $e\in R(K)$.
\end{proposition}

\begin{proposition}
\label{prop-Kostant-thm}
The space $\U\H(\g,\BK)$ is a free $\H(\g,\BK)$-module.  Moreover there is an isomorphism\[
\U\H(\g, \BK) \cong \bigoplus_{\ell \ge 0} \,\, W_\ell \otimes_{\C} \H(\g,\BK) ,
\]
of $K$-modules and $\H(\g,\BK)$-modules.
\end{proposition}

We shall prove these results in a moment, but first let us explain how they lead to a proof of Theorem~\ref{thm-spherical-classification}. 

If $\F$ is any algebraic family of Harish-Chandra modules, then the trivial isotypical summand    $e \F {=}\F^0 $ is   an $\H(\g, \BK)$-module and a  free $\O$-module.   The tensor product 
\[
\U\H(\g, \BK)\otimes _{\H(\g,\BK)} \mathcal{F}^0
\]
is a standard and spherical family, and it is determined up to isomorphism by the value of $TL^2 - R^2$, since by Proposition~\ref{prop-Hecke-algebra} this determines $\mathcal{F}^0$ as an $\H(\g,\BK)$-module, up to isomorphism. Now the formula $X\otimes f \mapsto Xf$ defines a \emph{classifying morphism} of families
\begin{equation}
\label{eq-classifying-morphism}
\U\H(\g, \BK)\otimes _{\H(\g,\BK)} \mathcal{F}^0 \longrightarrow \mathcal{F}.
\end{equation}
If $\F$ is itself standard and spherical, then \eqref{eq-classifying-morphism} is surjective, and, by counting dimensions in like isotypical summands using Proposition~\ref{prop-Kostant-thm}, we find that it is in fact an isomorphism.  So  there is at most one standard and spherical family, up to isomorphism, for any given value of $TL^2 -R^2$.

 In the reverse direction if $\mathcal{M}$ is any  $\H(\g,\BK)$-module  that is free and of rank one as an $\O$-module,  then the tensor product 
\begin{equation}
\label{eq-tensor-product-standard-fam}
\U\H(\g,\BK) \otimes_{\H(\g,\BK)} \mathcal{M}
\end{equation}
is  a standard and spherical family.  So using Proposition~\ref{prop-Hecke-algebra} again we can construct standard and spherical families with any given value of $TL^2 - R^2$.

Let us turn now to  the proofs of Propositions~\ref{prop-Hecke-algebra} and \ref{prop-Kostant-thm}.  The family $\g$ may be decomposed as a $K$-equivariant direct sum of $\O$-modules
\[
\g = \k \oplus \s , 
\]
where the summand $\s$ is freely generated as an $\O$-module by the elements   $R_1$, $R_2$ and $R_3$.  It will be convenient to denote by $\mathfrak{s}$ the $\R$-linear span of these three elements, so that 
\[
\s = \O\otimes_{\R} \mathfrak{s} .
\]
The three-dimensional real vector space $\mathfrak{s}$ carries the action \eqref{eq-O3-action} of the compact group $O(3)$.  It also carries an essentially unique invariant inner product, using which the symmetric and polynomial algebras of $\mathfrak{s}$ may be identified.

We shall use the following simple and well-known facts concerning the algebra of complex polynomial functions on  $\mathfrak{s}$.  First, the $K$-invariant polynomials are freely generated by $R^2 = R_1^2 {+}R_2^2 {+} R_3 ^2$.  Second, the full polynomial algebra is a free module over its $K$-invariant subalgebra.  To be more precise, if $\mathit{Harm}(\mathfrak{s})$ denotes the space of \emph{harmonic} polynomials on $\mathfrak{s}$, then the multiplication map
\[
\mathit{Harm}(\mathfrak{s}) \otimes \Sym(\mathfrak{s}) ^K \longrightarrow \Sym(\mathfrak{s})
\]
is a complex vector space isomorphism (this is a very special case of a much more general theorem of Kostant \cite{Kostant63,BernsteinLunts96}).

The degree $\ell$ harmonic polynomials form a copy of the  irreducible $SO(3)$-type representation of dimension $2\ell {+} 1$.  So the space $\mathit{Harm}(\mathfrak{s}) $  decomposes as a direct sum of precisely one copy of $W_\ell$ for each  $\ell=0,1,2,\dots$.

\begin{proof}[Proof of Proposition~\ref{prop-Hecke-algebra}]
Filter the algebra $\U(\g,\BK)$  by means of   the filtration \eqref{eq-tech-filtration} on the enveloping algebra and by deeming $R(K)$ to be of order zero.  The spherical Hecke algebra is a filtered subalgebra, and it suffices to show that its associated graded algebra is freely generated as an $\O$-algebra by the symbol of the order $2$ element $TL^2 - R^2$, or equivalently by the symbol of $R^2$.

The composition of the symmetrization map from the $K$-invariant part of the symmetric algebra of  $\s$ into $\U(\g)$ with mutliplication with    $e{\in} R(K)$ on both sides is  an isomorphism of filtered $\O$-modules 
\[
\Sym(\s)^K 
\stackrel\cong \longrightarrow \H(\g,\BK) .
\]
This is an algebra isomorphism at the level of associated graded algebras, because the elements in $\s\subseteq \U(\g)$ commute with one another up to elements of lower order.  Moreover the invariant part of the symmetric algebra   is freely generated by $R^2$, as we noted above. 
\end{proof}

\begin{proof}[Proof of Proposition~\ref{prop-Kostant-thm}]
Let us use the same filtration of $\mathcal{U}(\g,\BK)$ as in the previous proof, which we shall now restrict to $\mathcal{UH}(\g,\BK)$.  It suffices to show that the associated graded space is a free module over the associated graded algebra of the spherical Hecke algebra, and to compute that each $K$-type occurs in it with multiplicity one.

The associated graded space is $\Sym(\s)$, via the same identifications as in the previous lemma, while associated graded algebra of the spherical Hecke algebra is $\Sym(\s)^K$.  As we already noted, the former decomposes as  $\operatorname{Harm}(\mathfrak{s}) \otimes _{\C} \H(\g,\BK)$ as both a $K$-module and a module over the invariant functions.
\end{proof}

\subsection{Twisted dual families} 
\label{subsec-twisted-dual-families}
The same formula \eqref{eq-dual-family-K-rep} used for families of representations of $K$ defines the locally finite-dimen\-sional contragredient $\F^*$ of an algebraic family of Harish-Chandra modules.  But it will be a bit more relevant in this paper to examine the following variants of the locally finite-dimensional contragradient that involve the involutions introduced in Subsections~\ref{subsec-real-structure} and \ref{subsec-cartan}.

The \emph{$\theta$-twisted dual} $\F^{*,\theta}$ is simply $\F^*$ but with the action of $\g$ twisted by the involution $\theta$: for $\xi\in \F^*$ and $X \in \g$ we change the action of $X$ on $\xi $ to
\[
X \colon \xi \longmapsto \theta(X)\xi .
\]

The \emph{$\sigma$-twisted dual}  $\F^{*,\sigma}$  is the complex conjugate  space $\overline{\F^*}$, which we consider to be an $\O$-module via the given involution on $\O$: 
\[
p \cdot \overline{\xi} := \overline{\sigma(p)\xi} \qquad \forall p\in \O\quad  \forall  \xi \in \F .
\]
It is similarly a $\g$-module and a $K$-module via the involutions on $\g$ and $K$. See \cite[Sec. 2.4]{Ber2017} for further information.


\section{Principal Series Representations}
\label{sec-principal-series}

To provide some context for our treatment of families of Harish-Chandra modules in the next section, we shall review here the classification of the irreducible modules of a single Harish-Chandra pair 
$(\g\vert _\lambda , K)$ (see \cite[Ch.~3]{Wallach} for basic information concerning Harish-Chandra modules, which will in any case be reviewed in the families context in the next section).  We shall   focus on the case where $\lambda\ne 0$, and throughout this section $\lambda$ will be a fixed nonzero complex number.

\subsection{Parabolically induced representations}

The group $K{=}O(3)$ is disconnected, but it is the direct product of its connected identity component  $K_0{=}SO(3)$ and the two-element center of $K$.  The center acts trivially on $\g\vert _ \lambda $, and so the irreducible representations of the pair $(\g\vert _ \lambda , K)$ are partitioned into set of those where the center of $K$  acts trivially, and an otherwise identical set of irreducible representations where it    does not.  From now on we shall concentrate on the former, which correspond to irreducible $(\g\vert_ \lambda, K_0)$-modules.

We have seen that 
$(\g\vert _ \lambda , K_0) $ is isomorphic to $ (\mathfrak{so}(3,\C){\times}\mathfrak{so}(3,\C), SO(3))$, with $SO(3)$ acting diagonally.  Since the latter is the Harish-Chandra pair associated to the real reductive  group $PSL(2,\C)$,  we can determine the irreducible representations of $(\g\vert _ \lambda , K_0)$ by computing the irreducible admissible representations of $PSL(2,\C)$ up to infinitesimal equivalence.
  
The irreducible admissible representations of $PSL(2,\C)$, or more generally of $SL(2,\C)$, are well known and may be described using parabolic induction from the subgroup of upper triangular matrices.  Let $(\sigma,\varphi) \in \Z {\times} \C$. The parabolically induced representation of $SL(2,\C)$  associated to  the character
\[
\begin{bmatrix} a & 0 \\ 0 & a^{-1} \end{bmatrix} \longmapsto \operatorname{phase}(a)^\sigma |a|^{\varphi }
\]
of the diagonal subgroup is either irreducible or has a composition series with two irreducible factors, one finite-dimensional and one infinite-dimensional.  In either case, the representations, or irreducible factors, at $(\sigma,\varphi)$ and $(-\sigma,-\varphi)$ are the same.  See, for example \cite[Chap VIII, Problems 9-14]{Knapp86} for details concerning all the above.

Each infinite-dimensional, irreducible, admissible representation of the real reductive group $SL(2,\C)$   arises as  a parabolically induced representation, and does so in a unique way, except for the equivalence just mentioned. Moreover each finite-dimensional and   irreducible   representation of $SL(2,\C)$  arises as a composition factor of a unique reducible parabolically induced representation, again except for the equivalence between the factors associated to  $(\sigma,\varphi)$ and $(-\sigma,-\varphi)$. We find, in summary that the admissible dual of $SL(2,\C)$ is parametrized  by the set 
 \[
 \left ( \Z \times \C \right ) \big /\, \Z_2 ,
 \]
 where the group $\Z_2$ acts as the involution $(\sigma,\varphi) \mapsto (-\sigma,-\varphi)$.
The irreducible admissible representations of $SL(2,\C)$ that factor through the quotient  $PLS(2,C)$, and so correspond to irreducible Harish-Chandra modules for $(\g\vert_\lambda, K_0)$, are those with parameters in  the subset 
\[ 
 \left ( 2\,\Z \times \C \right ) \big /\, \Z_2 \subseteq  \left ( \Z \times \C \right ) \big /\, \Z_2 .
 \]

\subsection{Casimir elements}

In this paragraph we shall compute the the action of the central elements 
\begin{equation*}
L R \quad
\text{and}
\quad
 \lambda L^2 - R^2   
\end{equation*}
 in the enveloping algebra of $\g\vert _ \lambda $ on the principal series representation with parameters $(\sigma, \varphi)$. 

First, it  is well known that the $\mathfrak{so}(3)$ Casimir element
\begin{equation}
\label{eq-casimir-element}
\Omega  = L^2 =  L_1 ^2 + L_2 ^2 + L_3 ^2 
\end{equation}
acts as $-\frac 14 n(n+2)$ in the irreducible finite-dimensional representation of $\mathfrak{so}(3)$ of   dimension $n{+}1$ (this is the representation of highest weight $n$).  It follows from this that  if $\Omega^\pm$ are the Casimir elements associated to the generators $L _j^\pm$ in \eqref{eq-def-lambda}, then 
\[
\frac{1}{\sqrt{-\lambda} }LR = \Omega^+ - \Omega ^- = \pm \tfrac 14 \bigl ( m(m+2) - n(n+2)\bigr )
\]
and
\[
L^2 - \frac{1}{\lambda} R^2  = 2 \Omega ^+ +2  \Omega ^-
=  - \tfrac 12 \bigl ( m(m+2) + n(n+2)\bigr ) 
\]
in the representation of $\g\vert _ \lambda $ corresponding under the isomorphism \eqref{eq-splitting-gz} to the  finite-dimensional, irreducible,  tensor product representation of the direct product $\mathfrak{so}(3){\times}\mathfrak{so}(3)$ of highest weight $(m,n)$; the indeterminacy in sign corresponds to the indeterminacy in the choice of square root of $-\lambda$.

Now the above finite-dimensional representation of $(\g\vert_\lambda, K_0)$ arises as a quotient of the principal series representation with parameters
\begin{equation}
\label{eq-sigma-lambda-from-mn}
\sigma = m - n \quad \text{and} \quad  \varphi = 2 + m + n  .
\end{equation}
Since it follows from \eqref{eq-sigma-lambda-from-mn} that 
 \[
\sigma\varphi = m(m+2) - n(n+2)
\quad\text{and} \quad 
\sigma^2 + \varphi ^2 = 2\bigl ( m(m+2) + n(n+2) \bigr ) + 4 ,
\]
we find that  in the principal series representation with parameters $(\sigma,\varphi)$ as in \eqref{eq-sigma-lambda-from-mn}, the Casimir elements act as 
\begin{equation}
\label{eq-Casimirs-from-sigma-lambda}
L R = \pm  \sqrt{-\lambda }\frac{\sigma \varphi}{4}
\quad
\text{and}
\quad
 \lambda L^2 - R^2   
=  \lambda\frac{ 4 - \sigma^2 - \varphi ^2 }{4 } .
\end{equation}
Once again the indeterminacy in sign corresponds to the indeterminacy in the choice of square root, and hence the choice of isomorphism in \eqref{eq-splitting-gz}.\footnote{The outer automorphism of the pair $(\mathfrak{so}(3){\times}\mathfrak{so}(3), K_0)$ that flips the two $\mathfrak{so}(3)$-factors induces the involution  $(\sigma,\varphi)\mapsto (-\sigma, {\varphi})$ on parameters for irreducible modules.}  
 
Although we have computed the formulas  \eqref{eq-Casimirs-from-sigma-lambda} for $(\sigma,\varphi)$ satisfying \eqref{eq-sigma-lambda-from-mn} they must in fact hold for \emph{all} $(\sigma,\varphi)$ by a Zariski density argument, since both sides are regular in $\varphi$.

 \subsection{Spherical representations}
 
We are particularly interested in representations that include $K_0$-fixed vectors.  The lowest $K_0$-type in the principal series representation with parameters $(\sigma,\varphi)$ has highest weight $|\sigma|$, and so the   principal series  representations that include $K_0$-fixed vectors are precisely those for which $\sigma=0$.  Hence:
 
\begin{lemma}
The $K$-types of an irreducible infinite-dimensional representation with parameters $(0, \varphi)$ are the full set of $SO(3)$-types $\ell = 0,1,2,\dots$, and each occurs with multiplicity one. \qed
\end{lemma}

We shall again use the term \emph{standard} for representations with the above $K_0$-isotypical structure. 

We are also particularly interested in those representations in which the Casimir element $\lambda L^2 - R^2$ acts as the scalar $\lambda+\kappaconstant ^2$; compare Lemma~\ref{lem-vals-of-casimirs}.

 \begin{lemma}
\label{lem-parameters-for-our-reps}
Let $V$ be a standard representation of  $(\g\vert_\lambda , K_0)$. If the Casimir element 
$
 \lambda  L^2 - R^2 $ acts on $V$ as the scalar $  \lambda + \kappaconstant ^2 $,   then the irreducible quotient of $V$ including the $K_0$-type $\ell{=}0$ has classification parameters $\sigma = 0$ and $
\varphi^2 = - {4 \kappaconstant ^2} / {\lambda} $.
\end{lemma}

\begin{proof}
It follows from \eqref{eq-Casimirs-from-sigma-lambda} that 
\[
\lambda \cdot \frac{\varphi ^2-4 }{4 } + \lambda + \kappaconstant ^2 = R^2 - \lambda L^2 + \lambda + \kappaconstant ^2  = 0 ,
\]
or in other words that
\[
\lambda \cdot \frac{  \varphi ^2 }{4 } + \kappaconstant ^2   = 0 .
\]
This implies the result.
\end{proof}

\begin{corollary}
Let $V$ be a standard representation of  $(\g\vert_\lambda , K_0)$ in which the Casimir element 
$
 \lambda  L^2 - R^2 $ acts on $V$ as the scalar $  \lambda + \kappaconstant ^2 $.  Unless $- 4 \kappaconstant ^2/\lambda$ is a integer square, $V$ is an irreducible principal series representation.
\end{corollary}

\begin{proof}
The criterion for \emph{reducibility} of a parabolically induced representation is that $\varphi$ be an  integer of the same parity as $n$, with $|\varphi| > |n|$.   See \cite[Chap VIII, Problems 9-14]{Knapp86} again. So if $\varphi$ is non-integral, then any irreducible representation with parameters $(0,\varphi)$ is an irreducible principal series representation.  The result now follows from the Lemma~\ref{lem-parameters-for-our-reps}.
\end{proof}

\begin{remark}
In the next section we shall construct families  of representations  using solutions of the rescaled Schr\"odinger equation.  They will satisfy the hypotheses of the corollary, and as a result for almost all $\lambda$  the fibers will be irreducible principal series representations with parameters $(0, \sqrt{- 4 \kappaconstant ^2 / \lambda})$. Among other things our construction will  in effect extend the principal series to $\varphi {=} \infty$, where we will obtain a Harish-Chandra module for $(\g\vert _0 , K)$.
\end{remark}

\section{Solutions of the Schr\"odinger equation}
\label{sec-solutions}

In this section, we shall associate to every $\lambda\in \C$ a space $\RegSol (\lambda )$ of  ``regular'' solutions of the rescaled Schr\"odinger   equation $Tf = \eigv f$. Then we shall construct a natural algebraic family of Harish-Chandra modules $\RegSolFam$ whose fibers are the spaces $\RegSol(\lambda)$.

The solutions of the Schr\"odinger equation that are studied in physics are all regular, but here we shall also examine the ``singular'' solutions.  Contrary to what one might expect, the singular solutions will  play a decisive role in our study of the physical solutions of the Schr\"odinger equation.

\subsection{Solutions by separation of variables}

In this subsection we  shall review the standard approach to    the  solution of the   (rescaled) Schr\"{o}dinger equation by separation of variables in spherical coordinates.

\begin{defn} 
\label{def-sol-lambda}
For each $\eigv\in \C$ we  define
\[  
 \Sol(\eigv)=\bigl \{\, \psi\in C^{\infty}(\R^3_0)_\fin : T\psi=\eigv\psi \, \bigr \}.
\] 
\end{defn}

The Laplacian on $\R^3_0$ may be written as 
\[
\Delta = \frac{\partial ^2}{\partial r^2}+\frac{2}{r}\frac{\partial }{\partial r} 
+ \frac{1}{\phantom{.}r^2} \Omega
\]
where $\Omega$ is the Casimir operator \eqref{eq-casimir-element}   and where   $\partial/\partial r$ is differentiation in the radial direction: 
\[
\frac{\partial }{\partial r}  = \frac {1}{r} \left [ x _1 \frac{\partial }{\partial x_1} +x_2 \frac{\partial }{\partial x_2} +x_3 \frac{\partial }{\partial x_3} \right ] .
\]
So the rescaled Schr\"odinger operator $T$ acts on the $\ell$-isotypical component of the space $C^\infty(\R^3_0)_\fin$  as 
\begin{equation}
\label{eq-radial-fmla-for-T}
T = - \frac{\partial ^2}{\partial r^2} - \frac{2}{r}\frac{\partial }{\partial r} 
+ \frac{1}{\phantom{.}r^2} \ell(\ell+1) - \frac{2\kappaconstant }{r} .
\end{equation}
Using the standard diffeomorphism
\begin{equation}
\label{eq-spherical-diffeo}
\begin{aligned}
\R^3 _0  & \stackrel \cong \longrightarrow  \R_{+}\times S^2 
\\
x  & \longmapsto  
\bigl ( \, \|x\|\,  ,\,  \|x\|^{-1}x \,  \bigr ) ,
\end{aligned}
\end{equation}
we find that for any $\eigv\in \C$,  the $\ell$-isotypical space in $\Sol (\lambda)$ is the tensor product of the two-dimensional space of $\eigv$-eigenfunctions of \eqref{eq-radial-fmla-for-T} in $C^\infty (\R_{+})$ with the $\ell$-isotypical component of the space $C^\infty (S^2)_\fin$.  The latter is a single copy of the irreducible $SO(3)$-type representation   $W_\ell$; as noted in Section~\ref{Classification} it consists of the \emph{spherical harmonic functions} of degree $\ell$, which are the restrictions of harmonic polynomial functions of degree $\ell$ to the sphere.  We arrive at the following simple conclusion:

\begin{proposition}
For every $\lambda \in \C$ the subspace $\Sol(\lambda) \subseteq C^\infty (\R^3_0)_\fin $ is invariant under the  action of  the Lie algebra  $\g$.  Moreover the   action  of $\g$ on $\Sol(\lambda)$ factors through the fiber $\g\vert_\lambda$  and gives $\Sol(\lambda)$ the structure of an admissible $(\g\vert _\lambda , K)$-module, with each $SO(3)$-type representation of $K$ occurring with multiplicity $2$. \qed
\end{proposition}

\subsection{Regular  solutions of the Schr\"odinger equation}

Let us examine  the radial operator \eqref{eq-radial-fmla-for-T}, acting on smooth functions of $r>0$, a little more closely.  Its $\eigv$-eigen\-functions are precisely the  solutions of the differential equation
\begin{equation}
\label{equa3.1}
r^2 \frac{d^2 \psi }{dr^2} +2 r \frac{d\psi }{dr}    - \ell(\ell{+}1)  \psi  +  2 \kappaconstant  r \psi + \eigv r^2\psi  =0.
\end{equation}
This equation is regular-singular at $r=0$, with indicial roots $\ell$ and $-(\ell{+}1)$, and general theory (namely the Frobenius method; see \cite[Chapter X]{WhittakerWatson96} or \cite[Ch 4]{Teschl12})  provides a smooth solution defined on $\R_+$ that  in fact extends to an entire function of $r{\in} \C$ with   a zero of order precisely $\ell$ at $r=0$. 

\begin{defn}
\label{def-reg-solution}
We shall denote by  $F_{\ell,\lambda}$  the unique solution of \eqref{equa3.1} that extends to an  analytic function of $r\in \C$ and has the Taylor series 
\begin{equation}
\label{eq-regular-solution}
F_{\ell,\lambda} (r) = r^\ell + \text{higher-order terms in $r$}.
\end{equation}
\end{defn}

General theory also provides  a second, linearly independent solution of the form 
\begin{equation}
\label{eq-singular-solution}
G_{\ell,\lambda}(r) = H_{\ell,\lambda}(r) + C_{\ell,\lambda} \log(r)\, F_{\ell,\lambda}(r),
\end{equation}
 where 
$H_{\ell,\lambda}$  is a meromorphic function of $r$ with a pole of order $\ell{+}1$. The constant $C_{\ell,\lambda}$ might be $0$.  Whether or not that is the case,  the scalar multiples of $F_{\ell,\lambda}$ are the only solutions that are bounded near $0\in \R^3$.

\begin{defn} 
\label{def-RegSol}
We shall denote by $\RegSol(\lambda)$ the subspace of  $\Sol(\lambda)$ comprised of 
functions that are bounded near $0{\in} \R^3$.
Equivalently, $\RegSol(\lambda)$ 
is spanned by functions of the form
\[
\R_+\times S^2 \ni (r,\theta) \longmapsto F_{\ell,\lambda}(r) Y_\ell (\theta)
\]
in spherical coordinates \eqref{eq-spherical-diffeo}, where: 

\begin{enumerate}[\rm (i)]

\item 
  $F_{\ell, \lambda}$ is the regular solution of the radial equation \eqref{equa3.1}  with leading coefficient $1$, as in  \eqref{eq-regular-solution}, and  

\item $Y_\ell$ is any spherical harmonic function of degree $\ell$.
\end{enumerate}

\end{defn}

\begin{proposition}
\label{prop-regsol-is-rep}
The subspace $\RegSol(\eigv) \subseteq \Sol(\eigv)$ is a $(\g\vert _\lambda, K)$-submodule.
\end{proposition}

To prove this we shall use the following computation of the Runge-Lenz operators in polar coordinates.

\begin{lemma}
\label{lem-polar-form-of-R}
After applying  the diffeomorphism $\R^3 _0\cong \R_+{\times} S^2$ in  \eqref{eq-spherical-diffeo}  each of the Runge-Lenz operators $R \colon C^\infty (\R^3 _0)\to C^\infty (\R^3 _0)$ assumes the form
\[
\sqrt{-1} \cdot R  =   C  + \frac{1}{r} \cdot D  + \partial_r \cdot E , 
\]
where $C $, $D $ and $E $ are linear differential operators on $S^2$ with real coefficients. 
\end{lemma}

\begin{proof}
Each coordinate vector field $\partial_j$ on $\R^3_0$ has the form
\[
h_j\cdot \partial_r + \frac{1}{r} \cdot X_j,
\]
where $h_j$ is a smooth function on the $2$-sphere and $X_j$ is a vector field on $S^2$. The lemma follows from  this and the fact that 
\[
\sqrt{-1} \cdot R_i =  \sum_{j\ne i}\pm L_{i}\partial_j + \partial _i - \frac{\kappaconstant  x_i }{r} ,
\]
keeping in mind that each $L_{i}$ is a vector field on $S^2$ in spherical coordinates while $x_i/r$ is a smooth function on $S^2$.
\end{proof}

\begin{proof}[Proof of Proposition~\ref{prop-regsol-is-rep}]
Let 
$\psi = F_{\ell,\lambda} \cdot  Y_\ell $ be a generating element of the vector space $\RegSol(\lambda)$, as in Definition~\textup{\ref{def-RegSol}}.  
It suffices to show that  if $R$ is any one of the Runge-Lenz operators, then the function $R \psi$  is  bounded in a neighbourhood of $0\in \R^3$.

When $\ell{\ge} 1$,  boundedness is clear from the formula for the operator $R$ given in Lemma~\ref{lem-polar-form-of-R}, since $F_{\ell,\lambda}$ vanishes at $r{=}0$. But when $\ell{=}0$, the function $F_{\ell,\lambda}$ does not vanish at $r{=}0$, and the formula in Lemma~\ref{lem-polar-form-of-R} suggests that $R \psi$ might behave like $1/r$ at $r{=}0$. So a somewhat  closer examination is required.

Suppose then that $\ell{=}0$. Denote by $\langle R \psi\rangle^ k$ the projection of $R  \psi$ into the $k$-isotypical component of $\Sol(\lambda)$, so that 
\[
R \psi = \sum_{k\ge 0} \langle R  \psi \rangle^k .
\]
The functions $\langle R  \psi\rangle ^k$ for $k{\ge}1$ must be bounded near $r{=}0$, since the space $\Sol(\lambda)^k$ consists of functions that are either bounded near $r{=}0$ or have poles of order $k{+}1$, and $k{+}1\ge  2$, while $R \psi$ behaves at worst like $1/r$.  As for the component $\langle R  \psi\rangle^0$, it follows from the  definitions that 
\[
\langle R \psi\rangle ^0 = e R  e \psi ,
\]
and we have seen that  $e R  e$, like any element of
$\H(\g, \BK)$, acts on the $0$-isotypical component as a polynomial function of $\lambda$, so that $\langle R \psi\rangle ^0$  is a scalar multiple of $\psi$ and is therefore bounded near $r{=}0$ as required.\footnote{In fact one can compute that $e R  e{ =} 0$ in $\H(\g,\BK)$, so that $\langle R \psi\rangle ^0=0$.}
\end{proof}

\subsection{The algebraic family of regular solutions}

We shall now assemble the spaces $\RegSol(\lambda)$ into an algebraic family, as follows.

\begin{defn} 
We shall denote by $\RegSolFam$ the complex vector space that is  spanned by functions on $\C{\times} \R_{+}{\times} S^2$ (using spherical coordinates) of  the form 
\[
(\lambda, r,\theta) \longmapsto p (\lambda ) F_{\ell,\lambda} (r) Y_{\ell}(\theta) ,
\] 
where 
\begin{enumerate}[\rm (i) ]

\item $p $ is any polynomial function, and 

\item $F_{\ell, \lambda}$ and $Y_\ell$ are as in Definition~\ref{def-RegSol} above.

\end{enumerate}
We give $\RegSolFam$ the structure of a (free)  $\O $-module via the natural action of the rescaled Schr\"odinger operator $T\in \O$ on functions, which is in this case of course multiplication by  $\lambda$ since the action is on eigenfunctions.
\end{defn}

The group $K$ acts on $\RegSolFam$ as $\O $-module automorphisms via the action \eqref{eq-O3-action}  of $K$ on $\R^3_0$.  
There are obvious $K$-equivariant isomorphisms
\begin{equation}
\label{eq-regsol-fibers}
\RegSolFam \, \vert _\lambda \cong \RegSol(\lambda)
\end{equation}
for all $\lambda \in \C$.
We wish to  make  $\RegSolFam$ a $\g$-module via the natural action of $\g$ as differential operators on smooth functions on $\R^3_0$.  But we need to prove first that  $\RegSolFam$ is invariant under this action.  The following is one of our main results: 

\begin{theorem}
\label{thm-regsol-is-alg-family}
The vector space $\RegSolFam$ is invariant under the following action of elements $X\in \g$:
\[
X \colon \bigl ( p\cdot F_{\ell,\lambda}\cdot Y_\ell\bigr )
\longmapsto 
\Bigl [ (\lambda, r, \theta) \mapsto  p(\lambda) \bigl (X(F_{\ell,\lambda}Y_\ell)\bigr )(r,\theta)\Bigr ] .
\]
As a result, $\RegSolFam$ carries the structure of an algebraic family of Harish-Chandra modules.
\end{theorem}

It is enough to prove that $\RegSolFam$ is invariant under the action of the Runge-Lenz operators, and this is what we shall do in the next several lemmas.

\begin{lemma}
If $\ell \ge 1$ and if $\psi \in \RegSolFam^\ell$, and if $R$ is any Runge-Lenz operator, then $\langle R\psi \rangle ^{\ell{-}1}\in \RegSolFam^{\ell{-}1}$.
\end{lemma}

\begin{proof}
Think of  $\psi$ as a section $\lambda\mapsto \psi_\lambda$ of the family of   spaces $\RegSol(\lambda)$. 
We may assume that $\psi$ is one of the generators of  $\RegSolFam$, so that 
\[
\psi_\lambda = p(\lambda) \cdot F_{\ell,\lambda} \cdot Y_\ell .
\]
Using Lemma~\ref{lem-polar-form-of-R} we compute that 
\begin{multline}
\label{eq-R-applied-to-RegSol-section}
R\psi_\lambda = p(\lambda) \cdot \ell\cdot r^{\ell-1}\cdot D Y_{\ell} \,\, + \,\, 
p(\lambda) \cdot  r^{\ell-1}\cdot E Y_{\ell} 
\\ + 
\text{higher-order terms in $r$} .
\end{multline}
The formula shows that  the $(\ell{-}1)$-component of $R \psi$ belongs to $\RegSolFam$, as required.
\end{proof}

It will be convenient now to introduce a somewhat larger space than $\RegSolFam$.   We shall denote by $\RegSolFam_{\an}$ the space defined in exactly the same way as $\RegSolFam$, except that $p(\lambda)$ is replaced by any entire  function $h(\lambda)$.  Since the eigenfunctions $F_{\ell,\lambda}$ depend analytically on $\lambda$, the space  $\RegSolFam_{\an}$ is certainly invariant under the action of $\g$.

\begin{lemma}
\label{lem-integral-closure2}
If $\varphi\in \RegSolFam_{\an}$, and if $p \cdot \varphi \in \RegSolFam$ for some nonzero polynomial $p$, then $\varphi\in \RegSolFam$.
\end{lemma}

\begin{proof}
This follows from the fact that the only entire rational functions are polynomials.
\end{proof}

To proceed, let us make note of a simple fact that will be discussed further in the next section (and in any case the lemma may be proved by direct computation): 

\begin{lemma} 
\label{lem-generically-irreducible}
For all but countably many  $\lambda{\in} \C$ the space $\RegSol(\lambda)$ is an irreducible $(\g\vert _\lambda , K)$-module. \qed 
\end{lemma}

\begin{lemma}
\label{lem-induction-step}
If $\ell \ge 1$, if $\psi \in \RegSolFam^\ell_{\an}$, and if $\langle R S\psi \rangle ^{\ell{-}1}\in \RegSolFam^{\ell{-}1}$ for every Runge-Lenz operator $R$ and every $S\in R(K)$, then $\psi \in \RegSolFam^\ell$.
\end{lemma}
 
\begin{proof}
Assume that $\psi = h\cdot F_{\ell,\lambda}\cdot Y_{\ell}$, with $h$ holomorphic.
It follows from the explicit formula \eqref{eq-R-applied-to-RegSol-section} that  $\langle RS\psi\rangle^{\ell-1}$ has the form
\[
h\cdot  F_{\ell{-}1,\lambda} \cdot Z _{\ell-1}
\]
for some $Z_{\ell-1}$ in the $(\ell{-}1)$-isotypical part of $C^\infty(S^2)$ that depends  on the choices of $R$ and $S$. 

If $Y_{\ell}$ is nonzero, then $Z_{\ell{-}1}$  must be 
nonzero for some $R$ and some $S$, for otherwise the  $(\g\vert_\lambda, K)$-submodule of $\RegSol(\lambda)$ generated by $\psi_\lambda$ would be a proper submodule for every $\lambda$, contrary to Lemma~\ref{lem-generically-irreducible}.  In the case where $Z_{\ell{-}1}$ is nonzero, if $\langle RS\psi\rangle^{\ell-1}\in \RegSolFam$, then it follows that $h$ is in fact a polynomial function, and hence that $\psi\in \RegSolFam$, as required.
\end{proof}

\begin{lemma}
\label{lem-image-spherical-vectors}
If $X\in \U(\g,K)$ and if $\psi  \in \RegSolFam^0$, then $X \psi \in \RegSolFam$.
\end{lemma}

\begin{proof}
Fix $\psi \in \RegSolFam^0$.  We need to prove that $\langle X \psi\rangle ^\ell \in \RegSolFam$ for every $X$ and every $\ell\ge 0$.  If $\ell=0$, then since 
\[
\langle X \psi \rangle^0 = e   X  e \cdot \psi ,
\]
 with $e  X   e \in \mathcal{H}(\g,\BK)$, the required  result follows for every $X$  from Proposition~\ref{prop-Hecke-algebra} and Lemma~\ref{lem-vals-of-casimirs}, which together show that $e\cdot X \cdot e$ acts on $\RegSolFam^0$ as multiplication by a polynomial in $\lambda$.  The cases of higher $\ell$ are handled by induction on $\ell$ using Lemma~\ref{lem-induction-step}.
\end{proof}

\begin{lemma}
\label{lem-integral-closure1}
If $\psi \in \RegSolFam$, then there exists a
 nonzero polynomial $p$, some $S^0\in \mathcal{U}(\g,K)$ and some $\psi^0\in \RegSolFam^0$, such that  $p \cdot \psi = S^0 \cdot \psi^0$
\end{lemma}

\begin{proof}
It follows from the Peter-Weyl decomposition \eqref{eq-peter-weyl} and the $K$-iso\-typical decompositon of $\RegSolFam$ that there is a family of pairwise orthogonal idempotents $f_\alpha$ in $R(K)$ such that 
\[
\RegSolFam = \bigoplus_{\alpha} f_\alpha \cdot \RegSolFam
\]
and such that each summand on the right is a free $\O$-module of rank one.  It suffices to prove the lemma for generators of each of these summands, so let $\psi$ be a generator for the $\O$-module $f_\alpha\cdot \RegSolFam$.  Thanks to Lemma~\ref{lem-generically-irreducible} we can write 
\[
\psi_\lambda= (f_\alpha \cdot \psi)_\lambda = (f_\alpha \cdot S_\alpha \cdot   \psi^0)_\lambda
\]
for at least one value of $\lambda$, for some $S_\alpha \in \U(\g,\BK)$ and for some $\psi^0\in \RegSolFam^0$.  This means, in particular that $f_\alpha \cdot S_\alpha\cdot   \psi^0_\lambda$ is nonzero, and hence that $f_\alpha \cdot S _\alpha\cdot  \psi^0$ is nonzero.  Since the latter is an element of a free $\O$-module of rank one, of which $\psi$ is a generator, we can write 
\[
f_\alpha \cdot S_\alpha \cdot  \psi^0 = p \cdot \psi
\]
for some $p\in \O$, as required.
\end{proof}

\begin{proof}[Proof of Theorem~\ref{thm-regsol-is-alg-family}]
Given $\psi\in \RegSolFam$ and $S\in \mathcal{U}(\g,K)$, we wish to prove that $S\psi \in \RegSolFam$.    As we have already observed,  certainly $S \psi\in \RegSolFam_{\an}$.  But if $p \psi = S^0  \psi^0$, as in Lemma~\ref{lem-integral-closure1}, then 
\[
p S \psi = SS^0 \psi^0 ,
\]
and so it follows from Lemma~\ref{lem-image-spherical-vectors} that  $p S\psi \in \RegSolFam$. 
Lemma~\ref{lem-integral-closure2} now implies that $S\psi\in \RegSolFam$, as required. 
\end{proof}


\subsection{An explicit formula for the action}
\label{subsec-explicit-formulas}

It is also possible to prove Theorem~\ref{thm-regsol-is-alg-family} by an explicit computation, although the details are surprisingly intricate.  Here we give a sketch of the method and at the same time  give explicit formulas for a basis of the regular solutions.

To begin, if  $\eigv{\ne}0$, then it may be shown that 
\begin{equation}
\label{eq-SpecialFunc1}
F_{\ell,\lambda}(r)
= 
r^\ell e^{ -\sqrt{- \eigv}r}
M \Bigr  (-\frac{\kappaconstant}{ \sqrt{-\eigv}}{+}\ell{+}1 \, , \, 2\ell{+}2 \, , \, 2\sqrt{- \eigv}r\Bigl  )  ,
\end{equation}
 where $M(a,b,z)$ is Kummer's confluent hypergeometric function. 
 (It follows from   Kummer's   formula 
 \[ 
 e^{-z}M(a,b,z)=M(b-a,b,-z)
 \] 
 and the fact that Bessel functions of odd order are odd 
that $F_{\ell,\lambda}$ does not depend on the choice of the square root.) 
If $\eigv{=}0$, then 
 \begin{equation}
\label{eq-SpecialFunc2}
F_{\ell,\lambda}(r)
= 
\frac{(2\ell+1)!}{2^{\ell}\kappaconstant^{\ell}}\frac{1}{\sqrt{2\kappaconstant r}} \, J_{2\ell+1}(\sqrt{8 \kappaconstant r}) 
\end{equation}
where $J_{\nu}(z)$ is the Bessel function of the first kind.   See for example  \cite[Section 36]{LandauLifshitz}. 

Next, for  
 $\ell \ge m \ge -\ell$,   we define  $\psi^{\ell,m}\in \RegSolFam$ by 
\[
{\psi}^{\ell,m}_\eigv  
	=F_{\ell,\lambda}  \cdot Y^{\ell,m},
\]
where
the functions $Y^{\ell,m}$ are the basis for  the spherical harmonic functions of degree $\ell$ from  \cite[Chap.\ 15, eq.\ 15.137]{ARFKEN2013715}.   Now, form  the following combination of Runge-Lenz vectors:
\[
R_-=-iR_1 - R_2. 
\]
Then it may be calculated that 
\begin{multline}
\label{ac}
 R_-\psi^{\ell,\ell}_\eigv
	=
\sqrt{2\ell(2\ell +1)} \ell  \,\psi^{\ell-1,\ell-1}_\eigv\\
 - \frac{\sqrt{2}  } {\sqrt{(2\ell+3)(2\ell+1))}} \frac{\kappaconstant^2+{\eigv}(\ell+1)^2}{(\ell+1)(2\ell+3)}\,\psi^{\ell+1,\ell-1}_\eigv.
\end{multline}
The computation  involves the explicit form of the functions $F_{\ell,\lambda}$  along with various identities for hypergeometric, and other, functions.  

It can be shown that the entire $\g$-action is determined by this family of formulas.  The formulas indicate that $R_-$ maps $\psi^{\ell,\ell}$ back into $\RegSolFam$, and hence $\RegSolFam$  is invariant under the full action of $\g$.

\subsection{Singular solutions and the Wronskian form}
\label{subsec-wronskian}

In this subsection we shall examine the quotient
\[
\SingSol(\lambda) = \Sol(\lambda) / \RegSol (\lambda) 
\]
as a $(\g\vert_\lambda, K)$-module.

\begin{defn}
The (modified)
\emph{Wronskian} of two functions $\varphi$ and $\psi$ of $r>0$ is the function
\[
\wr (\varphi, \psi ) (r)  = r^2 \left (  \frac{d \varphi(r)}{d r} \psi (r) - \varphi(r)\frac{d  \psi (r)}{d r}\right ) .
\]
\end{defn}

The  modification ensures that if $\varphi$ and $\psi$ are two solutions of the radial equation \eqref{equa3.1}  for the same $\lambda$ and the same $\ell$, then $\wr (\varphi, \psi ) $  is a \emph{constant} function of $r$; to see this, differentiate with respect to $r$.   Moreover since all solutions are real-analytic functions of $r>0$,   the (modified) Wronskian   of two solutions vanishes identically if and only if the solutions are linearly dependent. 

 Of course the Wronskian  pays no attention to the angular coordinates in $S^2$.  In order to remedy this, we make the following definition.

\begin{defn} The  \emph{Wronskian form} on $\Sol(\lambda)$ is the bilinear form
\[
\Wr \colon \Sol (\lambda ) \times \Sol (\lambda) \longrightarrow \C
\]
defined by
\[
\Wr (\varphi, \psi) = \frac{1}{4 \pi} \int _{S^2} \wr (\varphi, \psi) \,\, d\operatorname{Area} .
\]
\end{defn}

The integral is a $K$-invariant and nondegenerate bilinear form on each $K$-iso\-typical subspace of $\Sol(\lambda)$, and moreover distinct isotypical subspaces are orthogonal to one another. 

\begin{lemma} If $\varphi,\psi \in \RegSol(\lambda)$, then $\Wr(\varphi,\psi) = 0$.
\end{lemma} 

\begin{proof}
If $\varphi$ and $\psi$ belong to distinct isotypical subspaces, then the integral is zero.  If they belong to the same isotypical space, then,  both functions being regular solutions,  their radial parts are linearly dependent, and so the integrand is zero.
\end{proof}
 
Thanks to the lemma, the Wronskian form defines a nondegenerate bilinear form 
\[
\Wr\colon \SingSol(\lambda) \times \RegSol(\lambda)\longrightarrow \C.
\]
We shall prove the following result:

\begin{theorem}
\label{thm-Wr-theta}
Let $\lambda \in \C$, and let 
 $X \in \g\vert _\lambda$. If  $\varphi \in \SingSol (\lambda)$ and $\psi \in \RegSol(\lambda)$, then
\[
\Wr( X \varphi, \psi ) + \Wr(\varphi, \theta(X) \psi) = 0 .
\]
As a result, the Wronskian  
induces an isomorphism of $(\g\vert _\lambda, K)$-mod\-ules  
from $\SingSol(\lambda) $ to the $\theta$-twisted dual of $\RegSol(\lambda)$. 
\end{theorem}
 
 \begin{corollary}
If $\lambda\in \R$, then the Wronskian form
induces an isomorphism of $(\g\vert _\lambda, K)$-mod\-ules  
from $\SingSol(\lambda) $ to the $\sigma$-twisted dual of $\RegSol(\lambda)$. 
\end{corollary}

\begin{proof}[Proof of the Corollary]
The theorem immediately gives
\[
\Wr( X \varphi, \overline\psi ) + \Wr(\varphi, \theta(X) \overline\psi) = 0 
\]
for $\varphi\in \SingSol(\lambda)$ and $\psi\in \RegSol(\overline{\lambda})$.  But this the same as 
\[
\Wr( X \varphi, \overline\psi ) + \Wr(\varphi, \overline{\sigma(X)\psi}) = 0 ,
\]
from which the result follows.
\end{proof}

We shall prove the theorem in a sequence of steps.  The first concerns the constituents of the formula 
\begin{equation}
\label{eq-fmla-for-r-recap}
\sqrt{-1} R = C + \frac 1r \cdot D + \partial _r \cdot E 
\end{equation}
for a Runge-Lenz operator that we obtained in Lemma~\ref{lem-polar-form-of-R}.    We shall use the following notation: if $A$ is a linear operator on a locally finite-dimensional representation space $V$ for $K$, then we shall denote by 
\[
A_k^{\ell}
\colon V^k \longrightarrow V^\ell
\]
the composition of the restriction of $A$ to the $k$-isotypical space of $V$ with the projection onto the $\ell$-isotpical space.
\begin{lemma}
If $D$ and $E$ are the operators in \eqref{eq-fmla-for-r-recap}, acting on smooth functions on $S^2$, and if $\ell\ge 1$, then
\label{lem-components-of-D-and-E}
\[
 E_{\ell-1}^\ell = (1{-}\ell)  D_{\ell-1}^\ell 
\]
and 
\[
 E_\ell^{\ell-1} = (1{+}\ell)  D_{\ell}^{\ell-1} .
\]
\end{lemma}

\begin{proof}
A short computation reveals that 
\begin{multline*}
[T,R] 
	= 
\frac{1}{r^2} \pr \cdot  \bigl ( 2 E - 2 D - [\Omega, D] \bigr )
- \frac 1 {r^2} \cdot \bigl ( \kappaconstant  D + [\Omega, C]\bigr )
\\ +\frac {1}{r^3} \cdot \bigl ( 2 [ \Omega, D] - 2 \Omega D -[\Omega, E] \bigr) .
\end{multline*}
Since $[T,R]=0$, each of the three terms in parentheses is zero.  Our interest is in the first. Since $\Omega$ acts as $\ell(\ell{+}1)$ on the $\ell$-isotypical space we obtain 
\[
\bigl [  2E - 2D +\ell(\ell{+}1) D -D(\ell{-}1)\ell\bigr ] _{\ell-1}^\ell = 0
\]
and after simplification this becomes
\[
\bigl  [ E - (1{-}\ell)D\bigr ] _{\ell-1}^\ell =  0
\]
as required for the first identity.  The second is handled similarly.
\end{proof}

\begin{lemma}
If $Y$ and $Z$ are smooth functions on $S^2$, then 
\[
\int _{S^2} Y\cdot DZ\, d\operatorname{Area}  = - 
\int_{S^2 } DY\cdot Z \, d\operatorname{Area}   
\]
and
\[
\int _{S^2} Y\cdot EZ\, d\operatorname{Area}  =  
\int_{S^2 } EY\cdot Z \, d\operatorname{Area}   
 - \, 
2 \int_{S^2 } DY\cdot Z \, d\operatorname{Area}   . 
\]
\end{lemma}

\begin{proof}
 The operator $\sqrt{-1} R$   is formally self-adjoint for the usual inner product of  functions on $\R^3$.  The present lemma follows from the formula for $\sqrt{-1}R$  and  from the fact that the formal adjoint of $\partial_r$ is $-\partial_r - 2/r$.
\end{proof}

\begin{lemma} 
\label{lem-tech-wronskian}
If $\varphi\in \SingSol(\lambda) ^{\ell-1} $ and $\psi\in \RegSol(\lambda)^\ell$, 
then 
\[
\Wr( R \varphi, \psi )  - \Wr (\varphi, R  \psi) = 0 .
\]
\end{lemma}

\begin{proof}
First, we compute that 
\begin{multline*}
\wr(R\varphi, \psi) = (2\ell^2 {+}\ell) r^0 \cdot DY\cdot Z   - (2 \ell{+}1) r^0 \cdot EY\cdot Z
+\,\, \text{positive-order terms in $r$} ,
\end{multline*}
for some spherical harmonic functions $Y$ and $Z$ of degrees $\ell{-}1$ and $\ell$, respectively.  Hence
\[
\Wr(R\varphi, \psi) = (2\ell^2 {+}\ell)  \int _{S^2} DY\cdot Z \, d\operatorname{Area}  
 - (2 \ell{+}1) \int _{S^2}  EY\cdot Z \, d\operatorname{Area}  ,
\]
and similarly
\[
\Wr(\varphi, R\psi) = -(2\ell^2 {-}\ell)  \int _{S^2} Y\cdot DZ \, d\operatorname{Area}  
 - (2 \ell{-}1) \int _{S^2}  Y\cdot EZ \, d\operatorname{Area}  .
\]
It follows that 
\begin{multline*}
\Wr(R\varphi, \psi)  - \Wr (\varphi, R \psi) 
= (2\ell {-}2)  \int _{S^2} DY\cdot Z \, d\operatorname{Area}  
 - 2 \int _{S^2}  EY\cdot Z \, d\operatorname{Area} , 
\end{multline*}
and to complete the proof we need only apply Lemma~\ref{lem-components-of-D-and-E}.
\end{proof}

\begin{proof}[Proof of Theorem~\ref{thm-Wr-theta}]
The Wronskian pairing determines a $K$-equivariant vector space isomorphism
\begin{equation}
\label{eq-potential-intertwiner}
\SingSol(\lambda) \longrightarrow \RegSol(\lambda)^{*,\theta}
\end{equation}
(the right-hand side is the $\theta$-twisted dual). Moreover this isomorphism is compatible with the action of the Runge-Lenz operators $R$ from $(\ell{-}1)$- to $\ell$-isotypical components, for all $\ell{>}0$, thanks to Lemma~\ref{lem-tech-wronskian}. 

Now, for  all but countably many $\lambda$, the $(\g\vert_\lambda, K)$-modules in \eqref{eq-potential-intertwiner} are irreducible, and moreover they are isomorphic to one another. This follows from the discussion in Section~\ref{sec-principal-series}. It therefore follows from Schur's lemma that for such $\lambda$, \eqref{eq-potential-intertwiner} can differ from an isomorphism of $(\g\vert _\lambda, K)$-modules only by a multiplicative scalar in each $K$-isotypical component, since these components have multiplicity one.  But compatibility with the Runge-Lenz operators, even to the limited extent we have exhibited it, implies that these scalars must all be equal to one another.  So \eqref{eq-potential-intertwiner} is indeed an isomorphism of $(\g\vert _\lambda, K)$-modules.  This proves the lemma when the modules in \eqref{eq-potential-intertwiner} are irreducible. 
The general case follows by a continuity argument.
\end{proof}

 \subsection{The algebraic family of singular solutions}
\label{subsec-singsolfam}
In view of the  preceding section it is natural to organize the modules $\SingSol(\lambda)$ into an algebraic family.   Linearly independent solutions to the radial Schr\"odinger equation are given in \eqref{eq-regular-solution} and \eqref{eq-singular-solution}, and from the formulas it is clear that the coefficient of $r^{-(\ell+1)}$ in a solution depends only on the class of the solution in $\SingSol(\lambda)$.

\begin{defn} 
We shall denote by $\SingSolFam$ the complex vector space spanned by sections
\[
\C \ni \lambda \longmapsto \phi_\lambda \in \SingSol(\lambda)
\]
that are representable in the form
\[
\psi_\lambda = p(\lambda)G_{\ell,\lambda}\cdot Y_\ell \in \Sol (\lambda)
\]
where $p(\lambda)$ is a polynomial function of $\lambda$ and where $G_{\ell,\lambda}$ is as in  \eqref{eq-singular-solution}, with coefficient  of $r^{-(\ell+1)}$ the constant $1$. \end{defn}

By repeating computations from the previous two subsections we arrive at the following result.

\begin{theorem}
\label{thm-singsol-is-a-twisted-dual}
The space $\SingSolFam$ is an algebraic family of Harish-Chandra modules for $(\g,\BK)$. The Wronskian form
\[
\SingSolFam \times \RegSolFam \longrightarrow \O
\]
induces an isomorphism of algebraic families from 
$
\SingSolFam $
to the $\theta$-twisted dual of $ \RegSolFam$. \qed
\end{theorem}

\begin{proposition}
The family  $\SingSolFam$ is a  standard and spherical family of $(\g,\BK)$-modules, in the sense of  \textup{Definition~\ref{defn-standard-spherical}}. 
\end{proposition}

\begin{proof}
We need to check   that  $\SingSolFam$ is generated by its spherical vectors.  It follows from the   formula in Lemma~\ref{lem-polar-form-of-R} for the Runge-Lenz operators in polar coordinates that if $\psi_\lambda = G_{\ell,\lambda}  Y_\ell$, then 
\begin{multline*}
R\psi_\lambda =  -(\ell{+}1)  r^{-(\ell{+}2)}\cdot D Y_{\ell} \,\, + \,\, 
  r^{-(\ell{+}2)}\cdot E Y_{\ell} 
+ 
\text{higher-order terms in $r$} .
\end{multline*}
Thus 
\[
\langle R\psi_\lambda  \rangle^{\ell+1} = G_{\ell{+}1,\lambda}  Z_{\ell+1}
\]
for some spherical harmonic function  $Z_{\ell+1}$.
Now for every $\ell{\ge}0$  there is at  least one $Y_\ell$, as above, for which $Z_{\ell+1}$ is nonzero, for otherwise every module $\SingSol(\lambda)$ would be reducible. It follows that  
\[
\SingSolFam^{\ell+1} \subseteq  \mathcal{U}(\g)\cdot   \SingSolFam^\ell,
\]
and hence that   $\SingSolFam$ is generated by its spherical vectors, as required.
\end{proof}

The following is an immediate consequence of the classification result in  Theorem~\ref{thm-spherical-classification} and    Lemma~\ref{lem-vals-of-casimirs}. 
 
\begin{theorem}
\label{thm-abstract-characterization-of-singsol} Up to isomorphism,
$\SingSolFam$ is the  unique standard and spherical family for which the Casimir element $TL^2 -R^2$ acts as $T + \kappaconstant ^2 I$.\qed
\end{theorem}
 
Of course, this result leads to a characterization of $\RegSolFam$, since it is the $\theta$-twisted dual of $\SingSolFam$.
  
\section{Physical Solutions and Jantzen quotients}\label{sec-phys-sol-jantzen}

In this section   we shall examine the solution spaces
\[
\PhysSol(\lambda)\subseteq \RegSol(\lambda)
\]
 for the Schr\"odinger equation that are used in physics.  The physical solutions    are characterized within $\RegSol(\lambda)$ by boundary conditions at infinity in $\R^3_0$, which imply that the physical solutions are either square-integrable, or nearly square-integrable.  The main purpose of this section is to explain that  the physical solutions can also be obtained in a purely represent\-ation-theoretic fashion using the Jantzen technique.

\subsection{Physical solution spaces}
\label{subsec-sol}
By the $K$-finite \textit{physical solutions} of the (re\-scaled) Schr\"odinger equation   we mean the solutions computed in physics texts by separation of variables in spherical coordinates.  See for example \cite[Section 36]{LandauLifshitz} or Section \ref{subsec-explicit-formulas}.

For a given $\eigv{\in} \R$,  the space $\PhysSol(\eigv)$ of physical solutions consists of  those functions in $\RegSol(\lambda)$  that converge to zero at infinity.   Computations using  explicit formulas  show that 
\begin{enumerate}[\rm (a)]

\item   $\PhysSol(- {\kappaconstant ^2}/{n^2})$ is the finite-dimensional subspace of $\RegSol(- {\kappaconstant ^2}/{n^2})$ comprised of the $SO(3)$-types $\ell {<} n$.  It is  a Harish-Chandra submodule.

\item   $\PhysSol(\eigv)=\RegSol(\lambda)$ for all $\lambda \ge 0$

\item $\PhysSol(\eigv) = 0$ otherwise.

\end{enumerate}
Moreover in every case the physical solution space has the structure of an irreducible $(\g\vert_\lambda, K)$-module and carries an in\-variant inner product.  See for example
\cite{Bander1},  \cite{Torres1998} and \cite{Bander2} for the cases of $\eigv{<}0$, $\eigv{=}0$ and $\eigv{>}0$, respectively.

\subsection{An intertwining operator}
\label{subsec-intertwining}

In this subsection we shall analyze the intertwining operator $\mathcal{A}$ from the algebraic family of singular solutions  to the regular of solutions that is characterized by the following diagram:
\begin{equation}
\label{eq-intertwiner-diagram}
\xymatrix{
\SingSolFam \ar[rr] ^{\mathcal{A}} & &\RegSolFam
\\
& \U\H(\g,\BK)\otimes _{\H(\g, \BK)} \O  \ar[ul]^{\cong} \ar[ur]& 
}
\end{equation}
The tensor product at the bottom is that for which the Casimir  $TL^2 {-} R^2{ \in} \H(\g, \BK)$ acts as  multiplication by $T {+} \kappaconstant ^2I$ on $\mathcal O$.
The diagonal morphisms are the classifying morphisms 
\eqref{eq-classifying-morphism} for $\SingSolFam $ and $\RegSolFam$ (we identify the $\ell{=}0$ isotypical parts of these families with $\mathcal{O}$ using any isomorphisms of $\H(\g,\BK)$-modules). The left classifying morphism is an isomorphism since the $\sigma$-twisted dual is standard and spherical; see   the proof of Theorem~\ref{thm-spherical-classification}.  There is therefore a unique morphism $\mathcal{A}$ that makes the diagram commute.

If $\mathcal{B}\colon \SingSolFam\to \RegSolFam$ is any other intertwining  morphism, then since almost all fibers in the families are irreducible, it follows from Schur's lemma that there are nonzero and relatively prime $p,q\in \O$ for which 
\[
p\cdot  \mathcal{A} = q \cdot \mathcal{B} .
\] 
But the morphism $\mathcal{A}$ is nonzero in each fiber since it is an isomorphism on $\ell{=}0$ isotypical spaces, and as a result $q$ must be a constant polynomial.   So, up to a nonzero complex scalar multiple, $\mathcal{A}$ is the unique intertwining morphism that is nonzero in each fiber.

We can identify all the $\ell$-isotyopical space $\RegSolFam^{\ell}$ as an $\O$-module and a $K$-module with $\O\otimes_{\C} W_\ell$, for example  by the mapping 
\[
\RegSolFam^{\ell} \ni F_{\ell,\lambda}\cdot Y_\ell \longmapsto 1{\otimes} Y_{\ell}\in \O\otimes_{\C} W_\ell.
\]
Compare Definition~\ref{def-RegSol}. We can do the same for the $\ell$-isotypical subspace of  $\SingSolFam $. After making these identifications,  it follows from Schur's lemma that  the restriction of $\mathcal{A}$  to the $\ell$-isotypical parts of our families is multiplication by a \emph{diagonal coefficient} polynomial function $\mathcal{A}_\ell(\lambda)$ that is determined up to multiplication by a nonzero constant. The constant depends on $\ell$, and varies with our choices of identifications.

The diagonal coefficient polynomials may be explicitly computed quite easily by making an explicit calculation with a single Runge-Lenz operator. The result is as follows:

\begin{proposition}
\label{prop-intertwiner}
The intertwining operator $\mathcal{A}$ has diagonal coefficients
\begin{equation*}
\mathcal{A}_{\ell}(\lambda)  = \text{\rm constant}_{\ell} \cdot \prod _{n=1}^\ell \left(\lambda n^2+\kappaconstant ^2 \right),
\end{equation*}
for all  $\ell \ge 1$. \qed
\end{proposition}

Compare \cite[Sec.~IV]{Subag}, where essentially the same computation is carried out for the $2$-dimensional hydrogen atom.

\subsection{Jantzen quotients}
\label{subsec-jantzen-review}
We begin with a quick introduction; see \cite[Sec.~4.1]{Ber2017} for further details, which are given  in precisely the context we are discussing. 

An intertwiner $\mathcal{A} \colon \F \to \H$ between algebraic families of Harish-Chandra modules over $\C$ determines \emph{Jantzen filtrations} of the fibers of $\F$ and $\H$, as follows: 
\begin{enumerate}[\rm (a)]

\item For $\F$, the filtration is decreasing, and for $p\in \Z$  the space $\F\vert_\lambda ^{(p)} $ is defined to be the image in the fiber $\F\vert_\lambda$ of all sections $f$ for which $\mathcal{A}(f)$ vanishes to order $p$ or more  at $\lambda$.  

\item For $\H$ the filtration is increasing, and  $\H\vert_\lambda^{(p)}$ is defined to be the image in the fiber  $\H\vert_\lambda$ of the space of all sections $h$ such that $(T-\lambda )^p \cdot h$ belongs to the image of $\mathcal{A}$.
\end{enumerate}
The map $f \mapsto (T-\lambda)^{-p} A(f)$ induces an isomorphism of $(\g\vert_\lambda, K)$-modules
\begin{equation}
\label{eq-jantzen-iso}
\mathcal{A}^{(p)}\colon \F\vert_\lambda ^{(p)} \Big / \F\vert_\lambda ^{(p+1)} \stackrel\cong \longrightarrow 
\H\vert_\lambda ^{(p)} \Big / \H\vert_\lambda ^{(p-1)} .
\end{equation}

For the specific intertwiner introduced in \eqref{eq-intertwiner-diagram} it  is a simple matter to read off the terms in the Jantzen filtration, and the \emph{Jantzen quotients}  \eqref{eq-jantzen-iso}, from Proposition~\ref{prop-intertwiner}.  We'll state the results  only for the regular solutions; the results for the singular solutions  can of course be inferred from the isomorphism in \eqref{eq-jantzen-iso}.

\begin{proposition}
Let $\lambda \in \C$.  The Janzten filtration of the fiber $\RegSol(\lambda)$ determined by the intertwiner \eqref{eq-intertwiner-diagram} is as follows:

\begin{enumerate}[\rm (a)]

\item If  $\lambda =  - \kappaconstant ^2 / n^2 $, then  $\RegSol(\lambda)^{(p)}$  is zero for $p{<}0$,  it is the direct sum of the $\ell$-isotypical subspaces in the range $\ell =  0,1,2, \dots,  n{-}1$ for $p{=}0$, and it is  the full space $\RegSol(\lambda)$ for $p{\ge} 1$. There are therefore exactly two nonzero Jantzen quotients.  The first  is $n^2$-dimensional with $K$-types 
 \[
 \ell =  0,1,2, \dots,  n{-}1 ,
 \]
  each with multiplicity one.  The second is irreducible and infinite-dimensional, with $K$-types
\[
 \ell =  n, n{+}1, n{+}2, \dots ,
 \]
each with multiplicity one. 
\item For other values of $\lambda$  the Jantzen filtration of $\RegSol(\lambda) $ is trivial: $\RegSol(\lambda)^{(p)}$ is zero for $p{<}0$, while it is the full space $\RegSol(\lambda)$ for $p{\ge} 0$.
\qed

\end{enumerate}
\end{proposition}

\begin{remark} 
The term \emph{Jantzen quotient} is standard, but of course a more accurate term would be  \emph{Jantzen subquotient}.  Indeed the Jantzen quotient of most interest  in the proposition above, namely the finite-dimensional Jantzen quotient in item (a), is in fact a submodule of $\RegSol(\lambda)$.
\end{remark}
\begin{corollary}
The negative part of the phsyical spectrum  of the rescaled Schr\"o\-dinger operator $T$  coincides with the set of all  $\eigv\in \C$ for which  $\RegSol(\lambda)$ has a nontrivial Jantzen filtration. \qed
\end{corollary}

\subsection{Non-degenerate invariant Hermitian forms}
\label{subsec-ndg-hermitian-forms}

Here we shall use the fact that $\SingSolFam$ is isomorphic to the $\sigma$-twisted dual of $\RegSolFam$ (this follows from Theorem~\ref{thm-spherical-classification}).  Intertwiners of the form
\[
\mathcal{A} \colon \RegSolFam^{*,\sigma}\longrightarrow \RegSolFam,
\]
between a family and its $\sigma$-twisted dual were analyzed in \cite[Sec.~4.2]{Ber2017}, where it was pointed out  that for a suitable $c\in \C^{\times}$ the formula 
\[
\bigl \langle \mathcal{A}^{(p)}[\alpha ], [\varphi ]\bigr \rangle =  c \cdot  \Bigl  ((T-\lambda )^{-p}\cdot \alpha (\varphi) \Bigr ) \Big\vert_\lambda 
\]
determines an invariant and nondegenerate hermitian form on the $p$'th Janzten quotient. Here $\alpha$ and $\varphi$ determine elements in the $p$'th parts of the Jantzen filtrations, as in Subsection~\ref{subsec-jantzen-review}, and $[\alpha]$ and $[\varphi]$ are the corresponding elements in the Jantzen quotients.  The term on the right is a polynomial in $T$, and in the formula we take its value   at $\lambda$.

The Hermitian  forms, like the Jantzen quotients themselves, are easy to compute using Proposition~\ref{prop-intertwiner}. Compare \cite[Sec.~5]{Ber2017} or \cite[Sec.~IV]{Subag}.  Doing so, one obtains  the following:

\begin{proposition}\label{propos5}
For $\lambda\in \R$ the above Hermitian form  is  definite  precisely on the following  Jantzen quotients:
\begin{enumerate}[\rm (a)]

\item 
The unique Jantzen quotient when $\lambda \ge 0$.

 \item 
The finite-dimensional Janzten quotient when 
$\eigv   = -{\kappaconstant ^2}/{n^2}$, for some natural number $n\ge 1$.\qed

\end{enumerate}
\end{proposition}

We arrive at the following theorem, which completely recovers both the physical spectrum of the rescaled Schr\"odinger operator and the physical solution spaces themselves from the algebraic family $\RegSolFam$.

\begin{theorem}\label{th3}
The physical  spectrum of the rescaled Schr\"odinger operator $T$ coincides with the set of all  $\eigv\in \R$ for which $\RegSol(\eigv)$  has  a nonzero infinitesimally unitary Jantzen quotient. There is in this case a unique infinitesimally unitary Jantzen quotient, and it  is the submodule     $\PhysSol(\lambda)$  of $\RegSol(\lambda)$.  \qed
\end{theorem}

\section{The Schr\"odinger Operator as a Self-Adjoint Operator}
\label{sec-spectral}

In this section we shall begin an examination of  the Schr\"odinger operator from the point of view of Hilbert space spectral theory by reviewing how one can obtain  a self-adjoint operator (in the technical sense of unbounded operator theory) from the rescaled Schr\"odinger operator
\begin{equation}
\label{eq-schrodinger-op2}
T=- \triangle-\frac{2 \kappaconstant  }{r} .
\end{equation}
There are two well-known approaches, both  leading to the same self-adjoint operator, and we shall need to make use of both of them in what follows.

\subsection{The Kato-Rellich theorem}

The    function $-2\kappaconstant  /r$ is locally squ\-are-int\-egrable on $\R^3$, and therefore the operator \eqref{eq-schrodinger-op2} is well-defined as map  from $C_c^\infty (\R^3)$ into $\mathcal{L}^2(\R^3)$.  Initially, we shall view $T$ as  a symmetric operator with this domain. 
The domain can then be extended, as follows.  Denote by $\mathcal{H}^2 (\R^3)$ the completion of $C_c^\infty (\R^3)$ in the norm 
\[
\| f \|^2 _{\mathcal{H}^2(\R^3)} =  \| \Delta f \|^2+ \| f\|^2 ,
\]
where---here and below---the unlabelled norms and  inner products are to be taken in $\mathcal{L}^2(\R^3)$.  This is a standard Sobolev space, of course. The inclusion of $C_c^\infty (\R^3)$ into $\mathcal{L}^2 (\R^3)$ extends to a continuous inclusion of $\mathcal{H}^2 (\R^3)$ into $\mathcal{L}^2 (\R^3)$. Moreover: 

\begin{lemma}
\label{lem-sobolev}
 The inclusion of $C_c^\infty (\R^3)$ into $\mathcal{L}^\infty  (\R^3)$ extends to a continuous inclusion of $\mathcal{H}^2 (\R^3)$ into $\mathcal{L}^\infty (\R^3)$.
\end{lemma}

This is a special case of the Sobolev embedding theorem.  As a result of the lemma, the operator \eqref{eq-schrodinger-op2} is in fact well-defined as a bounded linear operator from $\mathcal{H}^2(\R^3)$ into $\mathcal{L}^2(\R^3)$.
The following result is well known; see for example \cite[Chap.\ V, Thm.\ 5.4]{Kato76} for a proof.

\begin{theorem}
\label{thm-H-bounded-below-ess-s-a}
The rescaled Schr\"odinger operator $T$ for the 3-dimensional hydrogen atom is essentially self-adjoint  and bounded below on the domain $C_c^\infty (\R^3)$. The  domain of the unique self-adjoint extension is   the Sobolev space $\mathcal{H}^2 (\R^3) \subseteq \mathcal{L}^2 (\R^3)$.
\end{theorem}

\begin{remark}
\label{eq-operator-bounded-below}
Recall for later that to say that  $T$  is bounded below on $\mathcal{H}^2 (\R^3)$ is to say that there is a constant $C_0 \ge 0$ such that 
\begin{equation*}
\langle T f,f\rangle \ge  - C _0 \langle f,f \rangle 
\end{equation*}
for every $f\in \mathcal{H}^2( \R^3 )$.  Equivalently,    the spectrum of  the self-adjoint operator $T$ is bounded below by $-C_0$. 
\end{remark}

Theorem~\ref{thm-H-bounded-below-ess-s-a}  is the archetypical application  of the Kato-Rellich theorem, \cite[Chap.\ 5, Thm\ 4.4]{Kato76}, and accordingly we shall call the self-adjoint operator it provides the \emph{Kato-Rellich extension} of  \eqref{eq-schrodinger-op2}. 

\begin{theorem}
The spectrum of the Kato-Rellich extension of $T$ consists of the nonnegative reals $[0,\infty)$ together with a countable discrete subset of $(-\infty,0)$, each element of which is an eigenvalue of finite multiplicity.
\end{theorem}

This is proved by arguing that the potential term $-2\kappaconstant  /r$  in the rescaled Schr\"o\-dinger operator is a relatively compact perturbation of the kinetic term; see for example \cite[Chap.\ V, Thm.\ 5.7]{Kato76}.  

By elliptic regularity, the negative energy eigenfunctions for the Kato-Rellich extension are smooth functions on the open set $\R^3_0\subseteq \R^3$. They are, of course eigenfunctions, in the ordinary sense there, and they are square-integrable. Conversely, the space of all such square-integrable eigenfunctions may be determined explicitly, as we indicated in Section~\ref{subsec-sol}, and by Theorem~\ref{thm-H-bounded-below-ess-s-a} they all belong to the   domain of the Kato-Rellich extension. So after these explicit computations we arrive at the following:

\begin{theorem}
\label{thm-Kato-Rellich-spectrum}
The spectrum of the Kato-Rellich extension of the rescaled Schr\"od\-inger operator $T$ is the set 
\[\Bigl \{\, - \frac{\kappaconstant ^2 }{ n^2}   : n=1,2,\ldots \,  \Bigr\} \, \cup\, [0,\infty).
\]
\end{theorem}
\subsection{The Friedrichs extension}
In this section we shall review the Friedrichs extension of $T$. See for instance \cite[Chap.~VI, Sec.~3]{Kato76} for full details.

We have already seen that $T$ is bounded below on  $C^\infty _c (\R^3_0)$ (and indeed on a larger domain). Choose a constant $C_0 >  0$ so that 
\[
\langle Tf, f\rangle + C_0 \langle f, f\rangle \ge 0
\]
for every $f\in C_c^\infty (\R^3_0)$, and then choose $C {>}C_0$. Using this choice of $C$, denote by $\mathcal{F}(\R^3)$ the Hilbert space completion of $C_c^\infty (\R^3_0)$ in the norm associated to the inner product 
\begin{equation}
\label{eq-K-inner-product}
\langle f,g\rangle_{\mathcal{F}(\R^3)} = \langle (T+C)f,g\rangle_{\mathcal{L}^2 (\R^3)}   .
\end{equation}
There is a bounded linear map
\[
\mathcal{F} (\R^3) \longrightarrow \mathcal{L}^2 (\R^3)
\]
that extends the inclusion of $C_c^\infty (\R^3_0)$ into $\mathcal{L}^2(\R^3)$. The extended map is an inclusion, too, and we can therefore regard $\mathcal{F} (\R^3) $ as a dense linear subspace of $\mathcal{L}^2 (\R^3)$ as well as a Hilbert space in its own right.

\begin{theorem}
\label{thm-Friedrichs-ext}
There is a unique extension of the operator \eqref{eq-schrodinger-op2} with domain $C_c^\infty (\R^3_0)$ to a self-adjoint operator on a domain that is included within  the subspace  $\mathcal{F} (\R^3)\subseteq \mathcal{L}^2 (\R^3)$.  The extension is characterized by the formula
\[
\langle (T+C)^{-1}f,g\rangle _{\mathcal{F}(\R^3)}= \langle f,g\rangle_{\mathcal{L}^2 (\R^3)}
\]
for all $f\in \mathcal{L}^2 (\R^3)$ and all $g\in \mathcal{F}(R^3)$.
\end{theorem}

This is   the \emph{Friedrichs extension} of the rescaled Schr\"odinger operator $T$ from the initial domain $C_c^\infty(\R^3_0)$ to a self-adjoint operator that is bounded below; see  for example \cite[Chap.\ VI, Thm\, 2.11]{Kato76}.

\begin{theorem}
\label{thm-Friedrichs-is-Standard-extension}
The Friedrichs extension of the Schr\"odinger  operator from \textup{Theorem~\ref{thm-Friedrichs-ext}} coincides with the Kato-Rellich extension from \textup{Theorem~\ref{thm-H-bounded-below-ess-s-a}}.
\end{theorem}

\begin{proof} 
Thanks to the uniqueness assertion in Theorem~\ref{thm-Friedrichs-ext}, it suffices to prove that 
\[
\mathcal{H}^2 (\R^3) \subseteq \mathcal{F} (\R^3) \subseteq  \mathcal{L}^2 (\R^3).
\]
 We shall use the  Sobolev space $\mathcal{H}^1(\R^3)$, which  is the completion of $C_c^\infty (\R^3)$ in the norm associated to the inner product
\[
\langle f,f\rangle_{\mathcal{H}^1(\R^3)} =  - \langle \Delta f , f \rangle + \langle f,f\rangle .
\]
As with  $\mathcal{H}^2 (\R^3)$ and $\mathcal{F} (\R^3) $, the natural map from $\mathcal{H}^1 (\R^3 )$ into $\mathcal{L}^2(\R^3)$ is an inclusion, and moreover 
\[
\mathcal{H}^2 (\R^3)\subseteq \mathcal{H}^1 (\R^3) \subseteq \mathcal{L}^2 (\R^3).
\]
Now it may be shown without difficulty that $C_c^\infty (\R^3_0)$ is dense in $\mathcal{H}^1(\R^3)$, and since it is clear from \eqref{eq-K-inner-product} that 
\[
\langle f,f\rangle _{\mathcal{F}(\R^3)} \le \text{constant} \cdot \langle f,f \rangle _{\mathcal{H}^1(\R^3)}
\]
for all $f\in C_c^\infty (\R_0^3)$, the space $\mathcal{H}^1 (\R^3)$  is included within $\mathcal{F} (\R^3) $. We see therefore that  
\[
\mathcal{H}^2 (\R^3)\subseteq \mathcal{H}^1 (\R^3) \subseteq \mathcal{F} (\R^3),
\]
which proves the theorem.
\end{proof}

\section{Resolvents of the Schr\"odinger Operator}
\label{sec-resolvents}

In this section we shall work with the self-adjoint  rescaled Schr\"odinger operator  $T$,  with  $\domain(T) {=} \mathcal{H}^2 (\R^3)$, as in Section~\ref{sec-spectral}.  We shall use the resolvent operators $(T-\lambda)^{-1}$ to construct new families of eigenfunctions for $T$. Using these we shall give an explicit construction of  the  intertwining operator between the algebraic families of singular and regular solutions that we examined in Section~\ref{sec-phys-sol-jantzen}.

\subsection{Resolvents and solutions of the Schr\"odinger equation}
\label{subsec-resolvent-operators}

\begin{defn} 
A smooth function $\sigma\colon \R^3 \to \C$ is a \emph{cutoff  at $0$} if it is compactly supported and if it is  identically equal to $1$ in some neighborhood of $0$.   A smooth function $\sigma\colon \R^3 \to \C$ is a \emph{cutoff  at $\infty$} if it is identically equal to $1$ in the complement of some compact set.
\end{defn}

In the following definition $\domain(T)$ will refer to the self-adjoint domain of $T$ constructed in     Section~\ref{sec-spectral}.

\begin{defn}
We shall write 
\[
\domain_0 (T) = \bigl \{ \, f \colon \R^3 _0 \to \C : \text{$\sigma  f \in \domain(T)$ for some cutoff   $\sigma$ at $0$} \, \bigr \} 
\]
and
\[
\domain_\infty (T) = \bigl \{ \, f \colon \R^3 _0 \to \C : \text{$\sigma  f \in \domain(T)$ for some cutoff   $\sigma$ at $\infty$} \, \bigr \} .
\]
\end{defn}

  \begin{remark}
  Note that $\domain(T)$ is closed under pointwise multiplication by smooth functions on $\R^3$ that are constant outside of a compact set.  It follows that  $\domain(T)$ is included in both $\domain_0 (T)$ and $\domain_\infty (T)$.
\end{remark}

\begin{defn}
We shall write  
\[
\Sol_{0}(\lambda) = \Sol (\lambda) \cap \domain_0(T)
\quad \text{and} \quad \Sol_{\infty}(\lambda) = \Sol (\lambda)\cap \domain_\infty(T).
\]
\end{defn}

We shall be most concerned with $\Sol_{\infty}(\lambda) $, which according to the definition  consists of those $K$-finite smooth $\lambda$-eigen\-functions of the rescaled Schr\"odinger operator $T$   that   coincide with elements of $\domain(T)$ near infinity in $\R^3_0$.  Here is its most important property, as far as we are concerned:

\begin{theorem}
\label{thm-resolvent-solutions}
Suppose that $\lambda\notin \Spec(T)$. For every $G\in \SingSol(\lambda)$ there is a unique $V\in \Sol_\infty (\lambda)$ such that 
\[
\Wr (G,F) = \Wr (V,F)\qquad \forall F\in \RegSol(\lambda).
\]
\end{theorem}

\begin{proof}
First we shall prove existence, and for this it suffices to consider each $\ell$-isotypical component  of $\SingSol(\lambda)$ individually.  Suppose that  $G{\in} \SingSol(\lambda)$ is represented by the product of   a $\lambda$-eigenfunction $G_{\ell,\lambda}$ of the radial operator   \eqref{eq-radial-fmla-for-T} and a spherical harmonic function $Y_\ell$  of degree $\ell$.  Let $f$ be a smooth, compactly supported function on $\R^3_0$ of the form $\varphi\cdot Y_\ell$, where $\varphi$ is a smooth, compactly supported function of $r{\in} (0,\infty)$.  The function
\[
U = (T-\lambda)^{-1} f
\]
is smooth on $\R^3_0$ thanks to elliptic regularity, and has the form 
\[
U = \psi \cdot Y_\ell
\]
for some smooth function $\psi$ of $r$ (which is not necessarily compactly supported) thanks to the $K$-equivariance of the rescaled Schr\"odinger operator.  From 
the equation 
\[
TU  = \lambda U +f 
\]
we see that  $U$ is a $\lambda$-eigenfunction for $T$ outside of the compact set where $f$ is supported, and so $\psi$ is a $\lambda$-eigenfunction for the radial operator  \eqref{eq-radial-fmla-for-T} outside of a compact subset of $(0,\infty)$.

Now every  solution of the eigenvalue  equation  \eqref{equa3.1} that is defined on a subinterval of $(0,\infty)$ may be extended to a solution on $(0,\infty)$.  So there are $\lambda$-eigen\-functions $\psi_0$ and $\psi_\infty$ that that agree with $\psi$ near $0$ and near $\infty$, respectively.  The corresponding functions 
\[
U_0 = \psi_0 \cdot Y_\ell \quad \text{and} \quad U_\infty  = \psi_\infty \cdot Y_\ell 
\]
belong to $\Sol_0(\lambda)$ and $\Sol_\infty (\lambda)$, respectively. Note that since $(T-\lambda)^{-1}$ has dense range, there is at least one $\psi$ for which $U_0$ and $U_\infty$ are nonzero.

The functions  $\psi_0$ and $\psi_\infty$ cannot be linearly dependent, because if they were, then   after adjusting one by a multiplicative scalar  the two functions would glue together to determine a $\lambda$-eigenfunction in $\domain(T)$. This would contradict the assumption that $\lambda\notin \Spec (T)$.   So  $\psi_0$ and $\psi_\infty$ must span the two-dimensional space of  $\lambda$-eigen\-functions of the radial operator  \eqref{eq-radial-fmla-for-T}.  

Let 
   $F_{\ell,\lambda}$ be the regular radial eigenfunction in  \eqref{eq-regular-solution}, and let    $G_{\ell,\lambda}$ the  eigenfunction from  \eqref{eq-singular-solution}, normalized so that $\Wr(F_{\ell,\lambda},G_{\ell,\lambda}) =1 $.  From the above we can write 
\[
F_{\ell,\lambda}\cdot Y_\ell = a\cdot U_0 + b\cdot  U_\infty
\quad \text{and} \quad 
G_{\ell,\lambda}\cdot Y_\ell = c\cdot U_0 + d\cdot  U_\infty,
\]
for suitable constants $a,b,c,d$. The scalar $d$ must be nonzero because $G_{\ell,\lambda}$ is unbounded near $0$, and is therefore not an element of $\domain_0(T)$; compare Lemma~\ref{lem-sobolev}. It therefore follows from the second equation that $U_\infty$ is unbounded near $0$, and from this it follows from the first equation that   the scalar $b$ must be zero.  This implies that $a$ is nonzero, and so we can write 
\[
 d \cdot U_\infty =  G_{\ell,\lambda}\cdot Y_\ell - c\cdot U_0  = G_{\ell,\lambda}\cdot Y_\ell - c/a \cdot  F_{\ell,\lambda} \cdot Y_\ell .
\]
It follows that we can take $V =  d\cdot  U_\infty$. This completes the proof of existence.

To prove uniqueness, note first that the above argument shows $F_{\ell,\lambda}\cdot Y_\ell\in \Sol_0 (\lambda)$, and therefore $\RegSol(\lambda) \subseteq \Sol_0(\lambda)$. Now it suffices to show that if $V\in \Sol_\infty (\lambda)$, and if $\Wr(V,F) =0$ for all $F\in \RegSol(\lambda)$, then $V=0$. But if the Wronskians are all $0$, then $V\in \RegSol(\lambda)$, and hence $V\in \Sol_0(\lambda)$, so that in fact $V \in \domain(T)$. This again contradicts the assumption that $\lambda \notin \Spec(T)$.
\end{proof}

\begin{remark}
\label{rem-characterization-of-sol-zero-lambda-ell}
The proof shows that if $\lambda \notin \Spec(T)$, then
\[
\Sol_{0}(\lambda) = \RegSol(\lambda)
\quad \text{and} \quad 
\Sol (\lambda) = \Sol_0 ( \lambda) \oplus \Sol_\infty ( \lambda) .
\]
\end{remark}

\subsection{The family of Kodaira-type solutions}  
\label{subsec-kodaira-type-solutions}
In this subsection we shall use the solutions  in Theorem~\ref{thm-resolvent-solutions} to create an algebraic family of $(\g,\boldsymbol{K})$-mod\-ules.  We shall name the family after Kodaira, in recognition of his fundamental results from \cite{Kodaira49} that we shall review in the next subsection.

  \begin{theorem}
  \label{thm-action-of-runge-lenz-on-sol-infty}
  If $\lambda\notin\Spec(T)$,    if $V \in \Sol_\infty (\lambda)$, and if $R$ is any Runge-Lenz operator, then  $R V \in \Sol_\infty(\lambda)$.
 \end{theorem}
 
 \begin{proof}
The only issue is whether or not $RV\in \domain_\infty(T)$.  The smooth, $K$-finite  function $V$ is in the   Sobolev space $H^2 (\R^3) $  near infinity (that is, $\sigma V\in H^2(\R^3)$ for a suitable cutoff $\sigma$ at infinity).  But   $V$ is also an eigenfunction for $T$, so  the function   $ TV = - \Delta V - 2\kappaconstant / r V $  is in the   Sobolev space $H^2 (\R^3) $  near infinity. Since the Coulomb potential   has bounded derivatives away from zero, it follows from elliptic regularity that $V$ lies in  the  Sobolev space $H^4 (\R^3) $  near infinity. 
 
Now the  Runge-Lenz operators have the form
 \begin{equation}
\label{eq-form-of-runge-lenz-operator}
-{\sqrt{-1}} R_i=   \frac{ \kappaconstant   x_i}{  r}  - \partial_i +  \sum_{j\ne i}   \pm \partial_j L_{j}  .
\end{equation}
Each $L_j$ acts as a bounded operator on any $\ell$-isotypical part of $H^4 (\R^3) $.  The partial derivatives map $H^4(\R^3)$ into $H^3(\R^3)$, and so into $H^2(\R^3)$, and  $\kappaconstant  x_i/r$ maps $H^4(\R^3)$ into itself near infinity. \end{proof}
 
\begin{defn} 
\label{def-kodaira-sol-fam}
We shall denote by $\KodSolFam$ the complex vector space   spanned by functions   
\[
\bigl [ \C \setminus \Spec(T)\bigr ] \times \R^3_{0} \longrightarrow \C
\]
 of  the form 
$
(\lambda, x) \mapsto  V_\lambda(x) ,
$ 
where 
\begin{enumerate}[\rm (i) ]

\item $V_\lambda\in \Sol_\infty ( \lambda)$, and 

\item if $F\in \RegSolFam$, then $\Wr (F_\lambda,V_\lambda)$ is a polynomial function of $\lambda$ on the complement of $\Spec(T)$.

\end{enumerate}
We give $\KodSolFam$ the structure of an   $\O $-module via pointwise multiplication by polynomial functions in  $\lambda$, and the structure of a $K$-module via the actions of $K$ on the $K$-finite solution spaces $\Sol(\lambda)$.
\end{defn}

  \begin{lemma}  
  The  usual action of differential operators on functions gives the $\O$-module   the structure of a $(\g,\boldsymbol{K})$-module.
 \end{lemma} 
 
 \begin{proof}
 In view of Theorem~\ref{thm-action-of-runge-lenz-on-sol-infty} it only needs to be shown that if $V\in \KodSolFam$, and if $R\in \g$ is a Runge-Lenz operator, then  $\Wr (F_\lambda,RV_\lambda)$ is a polynomial function of $\lambda$.   But this follows from Theorem~\ref{thm-Wr-theta}.
 \end{proof}
 
 The following result shows that the family of Kodaira solutions carries the structure of an algebraic family of Harish-Chandra modules over $(\g,\BK)$, despite being initially defined only over the complement of the spectrum of $T$.
 
 \begin{theorem}
 \label{thm-iso-of-kodaira-and-sing-sol}
 The $\mathcal O$-module morphism
  \[
\SingSolFam \longrightarrow  \KodSolFam   
 \]
  that maps each $G\in \SingSolFam$ to the    unique $V \in \KodSolFam$ such that 
 \[
\qquad  \qquad \Wr (G_\lambda,F_\lambda ) = \Wr (V_\lambda ,F_\lambda ) \qquad  \forall F\in \RegSolFam \quad \forall \lambda \notin \Spec(T)
 \]
is an isomorphism of $(\g,\boldsymbol{K})$-modules.
 \end{theorem}
 
 \begin{remark}
 By definition, if $G\in \SingSolFam$, then $\Wr (G_\lambda,F_\lambda )$ is a polynomial function of $\lambda$. So the function $V$ defined above is indeed an element of $\KodSolFam$.
 \end{remark}

\begin{proof}
The $\g$-equivariance of the morphism follows from Theorem~\ref{thm-Wr-theta}.  The $K$-equivariance is clear. The fact that the morphism is an isomorphism follows from Theorem~\ref{thm-singsol-is-a-twisted-dual}.
\end{proof}

\subsection{The limiting absorption principle}
\label{subsec-limiting-absoprtion}
The title of this subsection refers to the idea that physical solutions of the Schr\"odinger equation for the real eigenvalue $\lambda$ should be expressible as limits, as $\varepsilon{\to} 0$, of  rapidly decaying $(\lambda{\pm}i\varepsilon)$-eigenfunctions (classically these correspond to solutions of equations with   damping terms added).  

In mathematics the term has come to refer to the convergence of resolvent operators $( T{-}\lambda )^{-1}$ in some topology as $\lambda$ converges to the spectrum, but here we shall keep close to the original meaning.  We shall use some decisive contributions of Kodaira \cite{Kodaira49}  on the limiting absorption principle to obtain an explicit intertwining operator from singular to regular solutions of the rescaled Schr\"odinger equation.  Later, in Theorem~\ref{thm-determination-of-spectral-measure},   we shall use the same techniques to fully determine the spectral theory of the rescaled Schr\"odinger  operator; this was one of Kodaira's original purposes.

The starting point is the following theorem of Kodaira concerning eigenfunctions for the rescaled Schr\"odinger operator associated to non-real eigenvalues.

\begin{theorem}[See   {\cite[Theorem~5.1]{Kodaira49}}]
\label{thm-kodaira-asymptotics}
If $\operatorname{Im}(k){\ge} 0$ and $k{\ne} 0$, then there is a unique $K$-invariant eigenfunction $U_k$    for the eigenvalue $k^2$ with 
\begin{equation*}
U_k(r) \sim \frac{1}{r}\exp\bigl (\sqrt{-1} kr + \frac{\kappaconstant \sqrt{-1}}{k} \log( r)\bigr ) \qquad \text{as $r \to \infty$}.
\end{equation*}
As $k$ varies, the  functions $U_k$ vary continuously in the  topology of uniform convergence of smooth functions and all their derivatives on compact subsets of $\R^3_0$.  
\end{theorem}

In the  numerator in  the display is the constant $\kappaconstant $ from \eqref{eq-kappa}.  Observe that if $\operatorname{Im}(k)>  0$, then $U_k \in \domain_\infty (T)$.

\begin{lemma}
\label{lem-wronskian-for-u} If $k{>}0$, then
$ \Wr (U_{k}, U_{-k}) = 2 i k$.  In particular, $U_k$ and $U_{-k}$ are linearly independent.\end{lemma}

\begin{proof}
Kodaira shows in  \cite[Theorem~5.1]{Kodaira49}    that differentiating both sides of the asymptotic equivalence in Theorem~\ref{thm-kodaira-asymptotics} with respect to $r$   preserves the asymptotic equivalence. The result follows.
\end{proof}

Now define a continuous function in the region  $\operatorname{Im}(k)\ge 0$ by 
\begin{equation}
\label{eq-kodaira-a}
a(k) = \Wr (F_{0,\lambda} , U_k) ,
\end{equation}
 where $\lambda = k^2$.

\begin{lemma}
\label{lem-kodaira-dmla-for-f} 
If $k>  0$, and if $\lambda = k^2$, then
\[
F_{0,\lambda}  = \tfrac{1}{2 i k} \bigl ( a(-k) U_k - a(k) U_{-k}\bigr ).
\]
\end{lemma}

\begin{proof}
It suffices to compute that the Wronskians of both sides of the asserted equality with $U_k$ and $U_{-k}$ are equal, since by Lemma~\ref{lem-wronskian-for-u} these two functions span the spherical $\lambda$-eigenfunctions.  
\end{proof}

\begin{lemma}
\label{lem-kodaira-a-is-nonzero1} 
If $\operatorname{Im}(k)>  0$  and $a(k){= } 0$, then $k^2 \in \Spec (T)$.
\end{lemma}

\begin{proof}
Let $\lambda {=} k^2$. We can    write   
\[
U_k = a(k) G_{0,\lambda}   +  b(k) F_{0,\lambda},
\]
where $F_{0,\lambda}$ and $G_{0,\lambda}$ are as in the proof of Theorem~\ref{thm-resolvent-solutions}, except that we view them here   as $K$-invariant functions on $\R^3_0$, and $b(k)$ is some    scalar.  If $a(k){=}0$, then $F_{0,\lambda}$ is a multiple of $U_k$, and hence belongs to $\domain_\infty(T)$.  It then follows from Remark~\ref{rem-characterization-of-sol-zero-lambda-ell} that $\lambda \in \Spec(T)$.
\end{proof}

\begin{lemma}
\label{lem-kodaira-a-is-nonzero2} 
If $k>  0$, then $a(-k) {=} \overline{a(k)}$ and  $a(k){\ne } 0$.
\end{lemma}

\begin{proof}
It follows from the uniqueness statement in Theorem~\ref{thm-kodaira-asymptotics} that if $k> 0$ then $\overline{U_k}$ = $U_{-k}$.  Since $F_{0,\lambda}$ is real-valued when $\lambda$ is real, it follows from Lemma~\ref{lem-kodaira-dmla-for-f} that $a(-k) = \overline{a(k)}$.  This shows that $a(k)$ must be nonzero, for otherwise $F_{0,\lambda}$ would be identically zero.
\end{proof}

\begin{theorem} 
\label{thm-limit-absorption}
Let $V_0\in \KodSolFam$ be the   family of $K$-invariant Kodaira-type eigenfunctions
such that 
\[
\Wr (V_{0,\lambda}, F_{0,\lambda})=1
\]
for every $\lambda\notin \Spec(T)$.
For every  $\lambda >0$ the limits 
\[
V^+_{0,\lambda} = \lim_{\varepsilon \to 0+} V_{0,\lambda+i\varepsilon}
\quad \text{and}\quad 
V^-_{0,\lambda} =  \lim_{\varepsilon \to 0+} V_{0,\lambda-i\varepsilon}
\]
exist in the standard topology on $C^\infty(\R^3_0)$ of uniform convergence on compact sets of functions and all their derivatives.  The limits are $\lambda$-eigenfunctions, and  moreover 
\[
V^+_{0,\lambda} - V^-_{0,\lambda}=  \Wr( V^+_\lambda ,  V^-_\lambda)\cdot F_{0,\lambda}  
\] 
for every $\lambda > 0$.
\end{theorem}

\begin{proof}
To begin, suppose that  $\lambda\notin \Spec (T)$. Choose $k{\in} \C$ such that  $\operatorname{Im}(k)>0$ and $k^2 = \lambda$. By Lemma~\ref{lem-kodaira-a-is-nonzero1}, $a(k)$ is nonzero.  As we already noted, $U_k\in \Sol_\infty (\lambda)$.  Therefore $V_{0,\lambda}$ and $U_k$ are multiples of one another, since the space of  $K$-invariant  functions in $\Sol_\infty (\lambda)$ is one-dimensional.  In fact it follows from \eqref{eq-kodaira-a} and the Wronskian condition in the statement of the theorem that  
\begin{equation}
\label{eq-rel-between-u-and-v}
V_{0,\lambda} = \frac{1}{a(k)} U_k
\end{equation}
Once again, this formula holds for $\lambda\notin \Spec (T)$ and $\operatorname{Im}(k) > 0$ with $k^2 = \lambda$.

Now suppose that $\lambda > 0$.  The existence of the limits in the statement of the theorem follows from the continuity properties of $U_k$ as a function of $k$, and from Lemmas~\ref{lem-kodaira-a-is-nonzero1} and \ref{lem-kodaira-a-is-nonzero2}.    In fact 
\begin{equation*}
V^{+}_{0,\lambda} =  \frac{1}{a(\sqrt{\lambda})} U_{\sqrt{\lambda}}
\quad\text{and}\quad 
V^{-}_{0,\lambda} = \frac{1}{a(-\sqrt{\lambda})} U_{-\sqrt{\lambda}} ,
\end{equation*}
where in both cases we used the positive square root of $\lambda$. 
The final formula in the statement of the theorem is a consequence of  Lemma~\ref{lem-kodaira-dmla-for-f}.
\end{proof}

For $\lambda > 0$ let us now define 
\begin{equation}
\label{eq-def-w-lambda}
w(\lambda) =  \Wr( V^+_{0,\lambda} ,  V^-_{0,\lambda}) = \frac{2i\sqrt{\lambda}\,\,}{|a(\sqrt{\lambda})|^2}
\end{equation}
as in Theorem~\ref{thm-limit-absorption}.

\begin{theorem} 
\label{thm-limiting-absorption2}
For every $V\in \KodSolFam$ and for every $\lambda >0$ the limits 
\[
V^+_\lambda = \lim_{\varepsilon \to 0+} V_{\lambda+i\varepsilon}
\quad \text{and}\quad 
V^-_\lambda =  \lim_{\varepsilon \to 0+} V_{\lambda-i\varepsilon}
\]
exist in the standard topology on $C^\infty(\R^3_0)$ of uniform convergence on compact sets of functions and all their derivatives.  The limits are $\lambda$-eigenfunctions, and  moreover
\[
V^+_\lambda - V^-_\lambda\in \RegSol(\lambda).
\]
In fact there is a unique $F\in \RegSolFam$ such that 
\[
V^+_\lambda - V^-_\lambda= {w(\lambda)}\cdot F_\lambda  
\]
for every $\lambda >0$.
\end{theorem}

\begin{proof}
The case where $V$ is spherical is covered by Theorem~\ref{thm-limit-absorption}.  Because the limits there exist in the $C^\infty$ topology, if $X \in \mathcal{U}(\g)$, then in addition the limits 
\[
(XV_{0,\lambda})^+ = \lim_{\varepsilon \to 0+} XV_{0,\lambda+i\varepsilon}
\quad \text{and}\quad 
(XV_{0,\lambda})^- =  \lim_{\varepsilon \to 0+} XV_{0,\lambda-i\varepsilon}
\]
exist, and    if $\lambda >0$, then 
\[
(XV_{0,\lambda})^+ - (XV_{0,\lambda})^-= w(\lambda) \cdot XF_{0,\lambda}  
\] 
This proves the theorem because $\KodSolFam$ is generated by its spherical isotypical component.
\end{proof}

We come now to one of the main results in this paper, which recovers the algebraic intertwining operator of Section~\ref{sec-phys-sol-jantzen} from analysis and the limiting absorption principle.

 \begin{theorem}
 \label{thm-lap-construction-of-intertwiner}
Define a morphism of $\O$-modules 
\[
\mathcal{A} \colon \KodSolFam\longrightarrow \RegSolFam
\]
by mapping $V \in \KodSolFam$ to the unique $F\in \RegSolFam$ for which 
\[
F _\lambda =  \frac{1}{w(\lambda) }  \bigr ( V^+_\lambda   - V^-_\lambda \bigr  )  \qquad \forall \lambda > 0,
\]
   as in Theorem~\textup{\ref{thm-limiting-absorption2}}. This  is a  morphism of  algebraic families of $(\g,\BK)$-modules, and an isomorphism on $\ell{=}0$ isotypical  components. \qed
\end{theorem}

Using the isomorphism in Theorem~\ref{thm-iso-of-kodaira-and-sing-sol}  we obtain from Theorem~\ref{thm-lap-construction-of-intertwiner} an intertwining operator from 
$ 
 \SingSolFam $ to $ \RegSolFam$. 
It agrees with the intertwiner analyzed in Subsection~\ref{subsec-intertwining} up to an overall complex multiplicative factor, because   up  to a scalar factor there is  a unique intertwiner that is an isomorphism on    $\ell{=}0$ isotypical  components. 

\section{The Spectral Decomposition of the Schr\"odinger Operator}
\label{sec-spectral-decomp}

The spectral theorem for unbounded self-adjoint operators  provides a \emph{measurable} family of solution spaces parametrized by the values  in the   spectrum of the Schr\"odinger operator.    These spaces, $\mathcal{H}_\lambda$, which are Hilbert spaces, carry actions of the visible symmetry group $K$. We shall prove that for almost every $\lambda$  in the spectrum, the space of  $K$-finite vectors in $\mathcal{H}_\lambda$  carries a natural action of the   hidden symmetry Lie algebra $\g\vert_\lambda$, and that the $(\g\vert_\lambda, K)$-module so-obtained coincides with the physical solution space for $ \lambda$.

\subsection{Direct integral decomposition} 

From now on we shall denote by $T$ the Kato-Rellich extension of the rescaled operator \eqref{eq-schrodinger-op2}. Since $T$ is self-adjoint, general theory provides a spectral decomposition of the underlying Hilbert space $\mathcal{L}^2(\R^3)$, as in the following theorem.  See for example \cite[A.69]{dixmier1982c} for the terminology and \cite[A.84]{dixmier1982c} for the result.

\begin{theorem}
\label{thm-dir-integral}
There is a fully-supported Borel measure $\mu$ on $\operatorname{Spec}(T)$, a $\mu$-measurable field of Hilbert spaces $\{ \,\mathcal{H}_\lambda\,\}_{\lambda \in \operatorname{Spec}(T)}$, and a unitary isomorphism 
\begin{equation*}
\mathcal{L}^2(\R^3)\stackrel \cong\longrightarrow  \int^{\oplus}_{\operatorname{spec}(T)}\mathcal{H}_{\lambda} \, d\mu(\lambda) ,
\end{equation*}
from $\mathcal{L}^2(\R^3)$ to the space of $\mu$-square-integrable sections of the measurable field, such that 
if $f\in \mathcal{L}^2(\R^3)$, then 
\[
[(T\pm i I)^{-1}f\,]_\lambda = (\lambda \pm iI)^{-1}[f]_\lambda
\]
for almost every $\lambda$ \textup{(}we are writing $\lambda \mapsto [f]_\lambda$ for the measurable, square-integrable section of $\{ \,\mathcal{H}_\lambda\,\}_{\lambda \in \operatorname{Spec}(T)}$ that is attached to $f\in \mathcal{L}^2(\R^3)$ by the unitary isomorphism\textup{)}.
\end{theorem}

The Hilbert spaces $\mathcal{H}_\lambda$ for $\lambda{<}0$ identify with the $\lambda$-eigenspaces for $T$ in $\mathcal{L}^2 (\R^3)$ via the isomorphism in the theorem, and so are easily analyzed. The situation for $\lambda{\ge}0$   is more complicated; it will be explained in the following subsections.

\subsection{Visible symmetries}

For the sequel, we shall fix a measure $\mu$ and a $\mu$-measur\-able field  as in   Theorem~\ref{thm-dir-integral}.  Our ultimate aim is to equip the fibers of the field with hidden symmetries, but we shall start here with the easier visible symmetries.

In the following theorem, by a measurable family of unitary actions of $K$ on the measurable field $\{ \mathcal{H}_{\lambda} \}$ we shall mean a family of unitary group actions, one on each of the fibers $\mathcal{H}_\lambda$, that maps measurable sections to measurable sections.

\begin{theorem}
There is a measurable family of unitary $K$-actions on the measurable field $\{ \mathcal{H}_\lambda \}$ such that 
\begin{equation*}
[g\cdot f]_\lambda = g\cdot [f]_\lambda 
\end{equation*}
for every $g\in K$ and almost every $\lambda$.
\end{theorem}

This is a  consequence of the fact that the $K$-action on $\mathcal{L}^2 (\R^3)$ maps the domain of the self-adjoint operator $T$ into itself, and fixes $T$. See \cite[8.3.1,18.7.6,A.80]{dixmier1982c}.

\subsection{Hidden symmetries}

Our   aim in this subsection is to decompose the Runge-Lenz operators 
\[
R \colon C_c^\infty (\R^3 _0) \longrightarrow C_c^\infty (\R^3 _0) .
\]
 into families of operators acting on the Hilbert spaces $\mathcal{H}_\lambda$  in Theorem~\ref{thm-dir-integral}, although actually it will be necessary to consider only the  $K$-finite vectors in $\mathcal{H}_\lambda$. We shall prove the following result:

\begin{theorem}
\label{thm-action-from-Friedrichs}
There are  operators 
\[
R_{\lambda} \colon \mathcal{H}_{\lambda,\fin}  \longrightarrow \mathcal{H}_{\lambda ,\fin}
\]
with the following properties:
\begin{enumerate}[\rm (a)]

\item if $\{ f_\lambda \}$ is a measurable section of $\{ \mathcal{H}_\lambda^\sigma \}$ for some $\sigma \in \widehat K$, then $\{ R_{\lambda}f_\lambda  \}$ is a measurable section of $\{\mathcal{H}_\lambda\}$; and

\item if $g\in C_c^\infty (\R^3_0)_\fin$, then $[ R g]_\lambda = R_{\lambda } [g]_\lambda $ for $\mu$-almost every $\lambda$.

\end{enumerate}

\end{theorem}

Well-known examples show that care must be taken when attempting to decompose one unbounded operator with respect to a direct integral decomposition of another; see for example \cite[Sec.\ 10]{Nelson59}.  So we shall proceed cautiously, following a sequence of small steps.

\begin{lemma}
Each of the Runge-Lenz operators on $C_c^\infty (\R^3 _0)_\fin $ extends to  a continuous linear operator from $\mathcal{F}(\R^3)_\fin$ into $\mathcal{L}^2(\R^3 )_\fin$ \textup{(}the operator is bounded on each isotypical subspace, but not uniformly bounded over all isotypical subspaces\textup{)}.  The extension   maps the subspace $\mathcal{H}^2 (\R^3)_\fin\subseteq \mathcal{F}(\R^3)_\fin$ into $\mathcal{F}(\R^3)_\fin$
\end{lemma}

\begin{proof}
From \eqref{eq-two-casimirs}  and Lemma~\ref{lem-vals-of-casimirs}
\begin{equation*}
     T L^2 - R^2  =    T + \kappaconstant ^2  I 
 \end{equation*}
on $C_c^\infty (\R^3 _0)_\fin$.  The operator $L^2$ in this formula is the Casimir for $K$, and on each isotypical component of $C_c^\infty (\R^3 _0)_\fin$ it acts as a scalar.  Therefore on any such  component we have
\begin{equation}
\label{eq-casimir-formula-for-R}
- \langle R^2  f ,f\rangle  =  \text{constant} \cdot   \langle Tf, f\rangle  + \text{constant}\cdot \langle f,f\rangle  .
\end{equation}
But recall that when viewed as unbounded operators on $\mathcal{L}^2 (\R^3)$ with domain $C_c^\infty (\R^3 _0)$, the Runge-Lenz operators are skew-symmetric.  So it follows from \eqref{eq-casimir-formula-for-R} that 
\[ 
\| R_1f\|^2 + \| R_2f\|^2  + \| R_3f\|^2  \le  \text{constant} \cdot  \| f\|^2 _{\mathcal{F}(\R^3)} .
\]
This proves the first assertion in the lemma. 

 To prove the second assertion, recall the formula \eqref{eq-form-of-runge-lenz-operator} for $R_i$.
The operators $L_{j}$ map each isotypical component of $\mathcal{H}^2(\R^3)$ boundedly into itself.  The operators  $\partial_i,\partial_j$ are bounded as operators from $\mathcal{H}^2 (\R^3)$ into $\mathcal{H}^1(\R^3)$, and hence as operators from $\mathcal{H}^2 (\R^3)$ into $\mathcal{F}(\R^3)$; see the proof  of Theorem~\ref{thm-Friedrichs-is-Standard-extension}. Finally, $\kappaconstant  x_i/r$ maps $\mathcal{H}^2 (\R^3)$ continuously into $\mathcal{H}^1(\R^3)$, and hence into $\mathcal{F}(\R^3)$.  This is a simple direct computation using Lemma~\ref{lem-sobolev}.
 \end{proof}

 \begin{lemma}
 \label{lem-little-Friedrichs-lemma1}
If $h\in \mathcal{F}(\R^3)$ and $g\in C_c^\infty (\R^3)$, then
 \[
  \langle  h , g \rangle _{ \mathcal{F}(\R^3)} =  \langle   h ,(T+C)g\rangle_{ \mathcal{L}^2(\R^3)}   .
\]
 \end{lemma}
 
 \begin{proof}
In the special case where $h\in C_c^\infty (\R^3_0)$, the  only difference between the right-hand side of the formula in the lemma and the definition \eqref{eq-K-inner-product} of the inner product on $\mathcal{F}(\R^3)$ is    the choice of factor in the inner product to which $T{+}C$ is applied. So the lemma follows in this special case from the symmetry of $T$ on $C_c^\infty (\R^3_0)\subseteq \mathcal{L}^2 (\R^3)$. Since both sides of the formula are continuous in $h\in \mathcal{F}(\R^3)$, the general case follows from the special case.
\end{proof}

 \begin{lemma}
 \label{lem-little-Friedrichs-lemma2}
If $h\in \mathcal{F}(\R^3)_\fin$ and $g\in C_c^\infty (\R^3)$, then
 \[
  \langle  R_i h , g \rangle _{ \mathcal{L}^2(\R^3)} =   -\langle   h ,R_ig\rangle_{ \mathcal{L}^2(\R^3)}   .
\]
 \end{lemma}
 
 \begin{proof}
 This is obvious for $h\in C_c^\infty (\R^3_0)$ and follows in general by continuity of the left-hand and right-hand sides in $h\in  \mathcal{F}(\R^3)_\fin$.\end{proof}
 
 \begin{lemma}
 \label{lem-little-commutation-lemma}
 The restriction of the operator
 \[
  R_i (T+C)^{-1} \colon \mathcal{L}^2(\R^3)_\fin \longrightarrow \mathcal{F}(\R^3)_\fin
  \]
 to $\mathcal{F}(\R^3)_\fin \subseteq \mathcal{L}^2 (\R^3)_\fin$  is equal to 
 $ (T{+}C)^{-1} R_i $ .
 \end{lemma}

\begin{proof}
It suffices to prove that 
\[
\langle   R_i (T+C)^{-1} f,g\rangle_{ \mathcal{F}(\R^3)} = \langle   (T+C)^{-1} R_i f,g\rangle_{ \mathcal{F}(\R^3)}  
\]
for all $f,g\in C_c^\infty (\R^3 _0)_\fin$.  Using Lemmas~\ref{lem-little-Friedrichs-lemma1} and ~\ref{lem-little-Friedrichs-lemma2} we compute that 
\[
\begin{aligned}
\bigl \langle   R_i (T{+}C)^{-1} f,g\bigr \rangle_{ \mathcal{F}(\R^3)} 
	& =\bigl\langle   R_i (T{+}C)^{-1} f,(T{+}C)g\bigr\rangle_{ \mathcal{L}^2(\R^3)}  \\
	&= - \bigl\langle    (T{+}C)^{-1} f,R_i(T{+}C)g\bigr\rangle_{ \mathcal{L}^2(\R^3)}  \\
	&= - \bigl\langle    (T{+}C)^{-1} f,(T{+}C)R_ig\bigr\rangle_{ \mathcal{L}^2(\R^3)}  \\
	&= - \bigl\langle    f, R_ig\bigr\rangle_{ \mathcal{L}^2(\R^3)}  \\
	&= \phantom{-} \bigl\langle    R_if,  g\bigr\rangle_{ \mathcal{L}^2(\R^3)} .  \\
\end{aligned}
\]
The proof is finished by using the characterization   in Theorem~\ref{thm-Friedrichs-ext}.
\end{proof}

\begin{proof}[Proof of Theorem~\ref{thm-action-from-Friedrichs}]
The operators 
\[
S_i =R_i(T+C)^{-1}
\]
 are defined on $\mathcal{L}^2 (\R^3)_\fin$ and are bounded on each isotypical component. Moreover it follows immediately from Lemma~\ref{lem-little-commutation-lemma} that these operators commute with the resolvent operator $(T{+}C)^{-1}$.  It follows that each $S_i$ decomposes as a bounded measurable family $\{ S_{i,\lambda}\}$ on each isotypical part of the measurable family  $\{ \mathcal{H}_\lambda\}$. The operators $S_{i,\lambda}$ have the property that if $f\in L^2 (\R^3)_\fin$, then
\begin{equation}
\label{eq-property-of-S-lambda}
  [S_if]_\lambda = S_{i,\lambda} [f]_\lambda
 \end{equation}
 for $\mu$-almost every $\lambda$.
Now we  define 
 \[
 R_{i,\lambda} = (\lambda {+} C)S_{i,\lambda} \colon \mathcal{H}_{\lambda,\fin} \longrightarrow \mathcal{H}_{\lambda,\fin} .
 \]
 If $g\in C_c^\infty (\R^3_0)$, and if $f = (T{+}C)g$, then \eqref{eq-property-of-S-lambda} implies that 
 \[
   [R_ig]_\lambda = R_{i,\lambda} [g]_\lambda,
   \]
 for $\mu$-almost every $\lambda$, as required.  
\end{proof}

\subsection{Gelfand-Kostyuchenko method}

The abstract direct integral decomposition in Theorem~\ref{thm-dir-integral} can be made more concrete using the theory of distributions. The technique, which sometimes called the Gelfand-Kostyu\-chenko method, is as follows.

\begin{theorem}
\label{thm-realization}
There is a family of continuous linear maps 
\[
\varepsilon_\lambda  \colon C_c^\infty (\R_0^3) \longrightarrow \mathcal{H}_\lambda
\qquad (\lambda \in \Spec(T) )
\]
such that 
\begin{enumerate}[\rm (a)]

\item
for every $f\in C_c^\infty (\R_0^3)$ the family $\{ \varepsilon_\lambda  (f)\} $ is a measurable and square-integrable section of the field $\{ \mathcal{H}_\lambda \}$; 

\item 
for every $f\in C_c^\infty (\R_0^3)$, the sections $\{ [f]_\lambda \} $ and $\{ \varepsilon_\lambda (f)\}$ are equal, almost everywhere; and 

\item for almost every $\lambda$ the map $\varepsilon_\lambda$ has dense range.

\end{enumerate}

\end{theorem}

See   \cite[Sec.\ 1]{Bernstein88} for an account of this result that is especially well suited to our purposes.

\begin{lemma}
\label{lem-eval-gK-mod}
The evaluation morphisms in Theorem~\ref{thm-realization} may be chosen so that they  are $K$-equivariant, and in addition so that 
\[
\varepsilon_\lambda (T f )  = \lambda \varepsilon_\lambda (f)\quad \text{and} \quad 
\varepsilon_\lambda (R_i f ) = R_{i,\lambda} \varepsilon_\lambda (f)
\]
for all $f\in C_c^\infty (\R^3_0)$ and all $i=1,2,3$.
\end{lemma}

\begin{proof}
Fix a countable dense set in $C_c^\infty (\R^3_0)$. For any element of this set, the displayed equations hold almost everywhere. So by redefining $\varepsilon_\lambda$ to be zero on a suitable nullset of values $\lambda$, we can arrange for the equations to hold for all $\lambda$ and all elements of our countable dense set.  The equations now hold for all $f$ by continuity.  The property of $K$-equivariance is handled in the same way.
 \end{proof}

 Now let us denote by  $\mathcal{D}'(\R_0^3)$ the space of distributions on $\R_0^3$; it is of course the topological dual space of  $C_c^\infty (\R_0^3)$.  It acquires by duality an action of the group $K$ and the Lie algebra $\g$.

\begin{defn}
\label{def-realization}
For decomposition $\{ \mathcal{H}_\lambda \}$ as in Theorem~\ref{thm-dir-integral} and a family of evaluation maps $\{ \varepsilon_\lambda\}$ as in Theorem~\ref{thm-realization} and Lemma~\ref{lem-eval-gK-mod}, the  
  \emph{realization morphisms} 
\[
\rho_\lambda \colon \mathcal{H}_\lambda \longrightarrow \mathcal{D}'(\R_0^3)
\]
are the complex conjugates of the adjoints of the morphisms $\varepsilon_\lambda$, and hence are characterized by 
\[
\langle  \varepsilon _\lambda (f) , v_\lambda \rangle _{\mathcal{H}_\lambda} 
=
\int_{\R^3 _0} \overline{f} \cdot \rho_\lambda (v_\lambda) ,
\]
where the integral on the right denotes the natural pairing between distributions on $\R^3 _0$ and test functions.
\end{defn}

\begin{lemma}
\label{lem-eigendistributions}
For almost every $\lambda\in \Spec (T)$, the realization  map 
\[
\rho_\lambda  \colon 
\mathcal{H}_\lambda 
\longrightarrow 
\mathcal{D}' (\R_0^3)
\]
is injective. For every  $\lambda$ the   map $\rho_\lambda$ is $K$-equivariant, and its image  lies within the space of $\lambda$-eigendistributions of the Schr\"odinger operator.  
\end{lemma}

\begin{proof}
Injectivity of $\rho_\lambda$ follows from the density of the image of $\varepsilon_\lambda$.  Equivariance and the eigendistribution property follow from the equivariance of $\varepsilon_\lambda$ and compatibility of $\varepsilon_\lambda$ with $T$, as in Lemma~\ref{lem-eval-gK-mod}.
\end{proof}

\subsection{Comparison with the physical solution spaces}

It is a consequence of  elliptic regularity and  Lemma \ref{lem-eigendistributions} above     that  the range of  the realization morphism in Definition~\ref{def-realization} in fact lies in the space of  smooth $\lambda$-eigen\-functions   for $T$ on $\R^3_0$, which we shall denote by $C^\infty (\R^3 _0)_\lambda$. Of course this space was studied in detail in Section~\ref{sec-solutions}, and we shall use some of the information presented there to  compute the $K$-equivariant map 
\[
\rho_\lambda \colon \mathcal{H}_{\lambda,\fin} \longrightarrow C^\infty (\R^3 _0)_{\lambda,\fin}
\]

\begin{theorem}
\label{thm-gk-module-on-spectral-fibers}
For almost every $\lambda $ there is an action of $\g\vert _\lambda  $ on the space $\mathcal{H}_{\lambda,\fin}$ such that 
\begin{enumerate}[\rm (a)]

\item the $\g\vert _\lambda  $-action on $\mathcal{H}_{\lambda , \fin}$, together with the $K$-action, gives $\mathcal{H}_{\lambda ,\fin}$ the structure of a $(\g\vert_\lambda, K)$-module; and  

\item the realization map $\rho_\lambda  \colon 
\mathcal{H}_{\lambda,\fin} 
\to 
C^\infty (\R_0^3)_{\lambda,\fin}$  is an injective morphism of $(\g\vert_\lambda, K)$-modules.
\end{enumerate}
\end{theorem}

\begin{proof}
The $K$-action and   Theorem~\ref{thm-action-from-Friedrichs} give actions of the individual generators $L_{i}$ and $R_i$ of $\g\vert_\lambda$ on $\mathcal{H}_{\lambda,\fin}$.  It follows from 
Lemma~\ref{lem-eval-gK-mod} that these actions commute with the  morphism $\varepsilon_\lambda$, and hence also with $\rho_\lambda$. It follows from the injectivity of $\rho_\lambda$  that the actions of the individual generators constitute an action of $\g\vert _\lambda$, and indeed a $(\g\vert _\lambda , K)$-module structure.
\end{proof}

\begin{theorem}\label{th7}
For almost every   $\lambda \in \Spec (T) $ the realization morphism 
\[
\rho_\lambda \colon  \mathcal{H}_{\lambda,\fin}
 \longrightarrow
C^\infty (\R^3 _0)_{\lambda,\fin}
\]
 is an isomorphism of $\g\vert _\lambda $-modules onto  
the physical solution space $\PhysSol(\lambda)$.
\end{theorem}

\begin{proof}
When $\lambda<0$ the theorem follows from the discussion preceding Theorem~\ref{thm-Kato-Rellich-spectrum}.  The singleton  set $\{ \lambda{=}0\}$ has $\mu$-measure zero since there is no $\lambda{=}0$ eigenfunction that is square-integrable.  So we can concentrate on the case where $\lambda > 0$. 

We have seen that $C^\infty (\R^3 _0)_{\eigv,\fin}$ is an extension
\[
0 \longrightarrow  
	\PhysSol(\lambda) \longrightarrow  
		 C^\infty (\R^3 _0)_{\lambda,\fin}  \longrightarrow  
		 C^\infty (\R^3 _0)_{\lambda,\fin} \big / \PhysSol(\lambda)\longrightarrow  
		0
\]
 of one irreducible $(\g\vert_\lambda, K)$-module by a second (which is a   copy of the first).  We shall show that the composition 
\begin{equation}
\label{eq-composition-is-zero-into-nonphys-space}
	\mathcal{H}_{\lambda,\fin} \longrightarrow 
		 C^\infty (\R^3 _0)_{\lambda,\fin}  \longrightarrow 
		 C^\infty (\R^3 _0)_{\lambda,\fin} \big / \PhysSol(\lambda)  
\end{equation}
is zero for almost every $\lambda$.  This  will show that the realization morphism maps the $(\g\vert _\lambda , K)$-module $\mathcal{H}_{\lambda,\fin} $   injectively into a submodule of $\PhysSol(\lambda)  $ for almost every $\lambda$.  Since the physical solution space is irreducible, when the realization morphism is injective the only possibilities are that the map is an isomorphism, or that $\mathcal{H}_{\lambda,\fin}$ is zero.  But since the spectrum of $T$ includes the full half line $(0,\infty)$, the space $\mathcal{H}_{\lambda,\fin}$, and hence $\mathcal{H}_\lambda$, can be zero only for a measure zero set of positive energy values.   So if we can show that \eqref{eq-composition-is-zero-into-nonphys-space} is zero, then the proof will be complete.

To prove that \eqref{eq-composition-is-zero-into-nonphys-space} is zero, since  the quotient $(\g\vert _\lambda ,K)$-module is irreducible it suffices to show that the restriction 
\begin{equation*}
	\mathcal{H}_{\lambda }^{0} \longrightarrow 
		 C^\infty (\R^3 _0)_{\lambda }^0\longrightarrow 
		 C^\infty (\R^3 _0)_{\lambda }^0\big / \PhysSol(\lambda)^0
\end{equation*}
to the trivial $K$-isotypical part  (that is,  the $\ell{=}0$ part) is zero.  This is what we shall now do. 

The space $\Sol(\lambda)^0{=}C^\infty (\R^3 _0)_{\lambda}^0$  is  span\-ned by the regular solution $F_{0,\lambda}$ in \eqref{eq-regular-solution}, which spans $\PhysSol(\lambda)^0$, and the solution $G_{0,\lambda}$ in \eqref{eq-singular-solution}. 
If $\{ v_\lambda \}$ is a measurable section of $\{ \mathcal{H}_\lambda^0\} $ then we may write 
\[
\rho_\lambda(v_\lambda) = a_\lambda F_{0,\lambda} + b_\lambda  G_{0,\lambda}
\]
where $a$ and $b$ are measurable scalar functions of $\lambda$.  We must  show that $b_\lambda=0$ for almost every $\lambda$.

Suppose that for some section $\{v_\lambda\}$ the function $b$ is not zero almost everywhere.  After multiplying the section by a suitable scalar function of $\lambda$, we may assume that $b_\lambda\ge 0$ for all $\lambda$; that $v_\lambda$ is square-integrable; that $v_\lambda$ vanishes outside a compact set in $(0,\infty)$; that $a$ and $b$ are   bounded functions; but   still $b$ is not zero almost everywhere.  The square-integrable section $\{ v_\lambda\}$ corresponds to a function  $g\in \mathcal{L}^2 (\R^3)$.  In fact the function  $g$ lies in the domain $\mathcal{H}^2 (\R^3)$ of the self-adjoint operator $T$ because the section $\{ \lambda\cdot v_\lambda \} $ is square-integrable.  But if $f\in  C_c^\infty (\R^3 _0)$, then 
\[
\begin{aligned}
\int_{\R^3} \overline{f(x)} \cdot g(x) \, dx 
	& = \int_{\operatorname{Spec}(T)} \langle \varepsilon_\lambda(f), v_\lambda\rangle \, d\mu (\lambda) 
\\
	& = \int_{\operatorname{Spec}(T)} \left ( \int _{\R^3 _0}   \overline{f} 
	\cdot  \rho_\lambda(v_\lambda) \,\right )  d\mu (\lambda) 
\\
	& = \int_{\operatorname{Spec}(T)} \left ( \int _{\R^3}   \overline{f(x)} 
	\cdot \bigl ( a_\lambda F_{0,\lambda} (x) + b_\lambda G_{0,\lambda} (x)\bigr )  \, dx \,\right )  d\mu (\lambda) 
\\
	& = \int _{\R^3}   \overline{f(x)} 
	\cdot \left ( \int_{\operatorname{Spec}(T)}  
	\bigl  ( a_\lambda F_{0,\lambda} (x) + b_\lambda G_{0,\lambda} (x)\bigr ) \,   d\mu (\lambda) \right )  \, dx .
	\end{aligned}
\]
This shows that the function $g$ is equal almost everywhere to the function 
\[
x \longmapsto \int_{\operatorname{Spec}(T)}  \bigl  ( a_\lambda F_{0,\lambda} (x) + b_\lambda G_{0,\lambda} (x)\bigr ) \,   d\mu (\lambda)  .
\]
But this is not an $\mathcal{L}^\infty$-function on $\R^3$, which contradicts Lemma~\ref{lem-sobolev}. \end{proof}

\subsection{The spectral measure}
\label{subsec-spectral-measure}
We conclude by   recalling for completeness  Kodaira's method   \cite{Kodaira49} for computing   the measure $\mu$ in the direct integral of Theorem~\ref{thm-dir-integral}.  We shall be brief; see the appendix of   \cite{HigsonTan20} for a more leisurely summary.

The starting point is the following formula of Kodaira and Titchmarsh for the spectral projection associated to any self-adjoint operator $T$ and any suitable interval $(\alpha,\beta)\subseteq \R$. 
\begin{equation}
  \label{eq-spectral-projection}
  P_{(\alpha,\beta)} 
  = \lim_{\varepsilon \to 0+} \frac{1}{2 \pi i }
  \int _{\alpha }^{\beta } 
  \, (T {-} \lambda {-} i\varepsilon )^{-1}   -  (T {-} \lambda  {+} i\varepsilon )^{-1}    
  \, d\lambda  .
\end{equation}
 This is valid  as long as neither $\alpha$ nor $\beta$ are   eigenvalues of $T$. 
The limit may be computed in our case using the following formula, which uses Kodaira's eigenfunctions $U_k$ from Theorem~\ref{thm-kodaira-asymptotics}: 

\begin{lemma}
\label{lem-greens-function}
Assume that  $\lambda\notin \Spec(T)$, and that $\lambda = k^2$, with $\operatorname{Im}(k)>0$.  If $f$ is any  smooth and compactly supported $K$-invariant function  on $\R^3_0$,   then 
\begin{multline*}
\Wr(F_{0,\lambda},U_{k})\cdot \bigl ( (T {-}\lambda) ^{-1} f\bigr )(r )    \\
	= F_{0,\lambda}(r) \int _r ^\infty U_{k} (s) f (s )\, ds + U_{k}  (r) \int_0^r F_{0,\lambda}   (s) f(s) \, ds .
\end{multline*}
\end{lemma}

\begin{proof}
If we denote by $(K_{\lambda}  f)(r)$ the right-hand side of the formula in the statement of the lemma, then the formula defines a linear operator $K_{\lambda} $ from smooth compactly supported functions on $(0,\infty)$ to  the vector space of smooth functions in $\operatorname{dom} (T)$.  We compute directly from the definition of the Wronskian that
\begin{equation*}
(  T -\lambda )K_{\lambda}  f   = \Wr(F_{0,\lambda},U_{k}) f  ,
\end{equation*}
It follows  that 
 $ K_{\lambda}  f 
 =\Wr(F_{0,\lambda},U_{k})\cdot (T {-} \lambda )^{-1} f $,  as required.
\end{proof}

 Now we apply the lemma with $\lambda \pm i\varepsilon$ in place of $\lambda$, with $\lambda >0$.
Using our previous notation $a(k)$ for the Wronskian in Lemma~\ref{lem-greens-function} we obtain 
 \begin{multline*}
\lim_{\varepsilon\to 0+} \bigl ( 
(T {-} \lambda {-} i\varepsilon )^{-1} f  -  (T {-} \lambda  {+} i\varepsilon )^{-1}f 
\bigr ) 
\\
\begin{aligned}
& = F_{0,\lambda}(r) \int _r ^\infty \tfrac{1}{a(k)}U_{k} (s) f (s )\, ds + \tfrac{1}{a(k)}U_{k}  (r) \int_0^r F_{0,\lambda}   (s) f(s) \, ds 
\\
& \quad +  F_{0,\lambda}(r) \int _r ^\infty \tfrac{1}{a(-k)}U_{-k} (s) f (s )\, ds + \tfrac{1}{a(-k)}U_{-k}  (r) \int_0^r F_{0,\lambda}   (s) f(s) \, ds ,
\end{aligned}
 \end{multline*}
where $k$ is the positive square root of $\lambda$. Taking into account Lemma~\ref{lem-kodaira-dmla-for-f} we obtain 
 \begin{equation*}
\lim_{\varepsilon\to 0+} \bigl ( 
(T {-} \lambda {-} i\varepsilon )^{-1} f  -  (T {-} \lambda  {+} i\varepsilon )^{-1}f 
\bigr ) 
= 
\frac{2i\sqrt{\lambda}\,\,} {|a(\sqrt{\lambda})|^2} F_{0,\lambda}(r) \int _0 ^\infty F_{0,\lambda}(s) f (s )\, ds .
\end{equation*}
We find that on spherical functions that are smooth and compactly supported in $\R^3_0$, 
\begin{equation*}
  \label{eq-spectral-projection2}
  ( P_{(\alpha,\beta)} f)(r)
  =   \frac{1}{ \pi  }
  \int _{\alpha }^{\beta } \int_0^\infty F_{0,\lambda}(r)  F_{0,\lambda}(s) f (s )\, ds \frac{\sqrt{\lambda}\,\, }{|a(\sqrt{\lambda})|^2}d\lambda .
\end{equation*}
This computation is equivalent to the following result:

\begin{theorem}[Compare {\cite[Theorem 4.1]{Kodaira49}}]
\label{thm-determination-of-spectral-measure}
If the Hilbert spaces $\H_\lambda$ in \textup{Theorem~\ref{thm-dir-integral}} are normed so that $\|F_{0,\lambda}\| = 1$ for all $\lambda\in \Spec (T)$, then 
on $(0,\infty)$ the measure $\mu$ in \textup{Theorem~\ref{thm-dir-integral}} is given by the formula
\begin{equation*}
 2\pi i \cdot d\mu (\lambda ) = w(\lambda) d\lambda ,      
\end{equation*}
with $w(\lambda)$ as in \eqref{eq-def-w-lambda}.\qed
\end{theorem}

\begin{remark}
It is possible to   compute $w(\lambda)$ by determining the asymptotic behavior of  the  solutions $F_{0,\lambda}$; compare \cite[(36.27) \& (36.28)]{LandauLifshitz}.  The result is 
\[
w(\lambda) = 4 \pi i \kappaconstant (1 - e^{2 \pi \kappaconstant/\sqrt{\lambda}})^{-1}  \qquad (\lambda > 0).
\]
Observe this has singularities when $\lambda$ is one of the discrete spectral values $ - \kappaconstant^2 / n^2$ for $T$. The singularities come from zeros of $a (\sqrt{\lambda})$, as predicted by a formula discovered by Heisenberg  and  proved by Kodaira in \cite[Sec.~6]{Kodaira49}. Compare Lemma~\ref{lem-kodaira-a-is-nonzero1}.  There is a further essential singularity at $0$, even when $w$ is considered as a function of $k=\sqrt{\lambda}$.  This seems to be an inescapable   feature in scattering theory, and so  it is interesting that in the algebraic approach there is no kind of singularity at $0$ at all.
\end{remark}

\bibliography{references}

\def\cprime{$'$}
\begin{thebibliography}{TdCA98}

\bibitem[AWH13]{ARFKEN2013715}
G.~B. Arfken, H.~J. Weber, and F.~E. Harris.
\newblock {\em Mathematical methods for physicists}.
\newblock Academic Press, Boston, {S}eventh edition, 2013.

\bibitem[Bar36]{Bargmann}
V.~Bargmann.
\newblock {Zur Theorie des Wasserstoffatoms}.
\newblock {\em Z. Phys.}, 99:576--582, 1936.

\bibitem[Ber88]{Bernstein88}
J.~Bernstein.
\newblock On the support of {P}lancherel measure.
\newblock {\em J. Geom. Phys.}, 5(4):663--710 (1989), 1988.

\bibitem[BHS18]{Bar2017}
D.~Barbasch, N.~Higson, and E.~M. Subag.
\newblock Algebraic families of groups and commuting involutions.
\newblock {\em Internat. J. Math.}, 29(4):1850030, 18, 2018.

\bibitem[BHS20a]{Ber2016}
J.~Bernstein, N.~Higson, and E.~Subag.
\newblock Algebraic families of {H}arish-{C}handra pairs.
\newblock {\em Int. Math. Res. Not. IMRN}, (15):4776--4808, 2020.

\bibitem[BHS20b]{Ber2017}
J.~Bernstein, N.~Higson, and E.~Subag.
\newblock Contractions of representations and algebraic families of
  {H}arish-{C}handra modules.
\newblock {\em Int. Math. Res. Not. IMRN}, (11):3494--3520, 2020.

\bibitem[BI66a]{Bander1}
M.~Bander and C.~Itzykson.
\newblock Group theory and the hydrogen atom. {I}.
\newblock {\em Rev. Modern Phys.}, 38:330--345, 1966.

\bibitem[BI66b]{Bander2}
M.~Bander and C.~Itzykson.
\newblock Group theory and the hydrogen atom. {II}.
\newblock {\em Rev. Modern Phys.}, 38:346--358, 1966.

\bibitem[BL96]{BernsteinLunts96}
J.~Bernstein and V.~Lunts.
\newblock A simple proof of {K}ostant's theorem that {$U(\mathfrak{g})$} is
  free over its center.
\newblock {\em Amer. J. Math.}, 118(5):979--987, 1996.

\bibitem[Bou60]{Bourbaki60}
N.~Bourbaki.
\newblock {\em \'{E}l\'{e}ments de math\'{e}matique. {XXVI}. {G}roupes et
  alg\`ebres de {L}ie. {C}hapitre 1: {A}lg\`ebres de {L}ie}.
\newblock Actualit\'{e}s Sci. Ind. No. 1285. Hermann, Paris, 1960.

\bibitem[Dix82]{dixmier1982c}
J.~Dixmier.
\newblock {\em C*-algebras}.
\newblock North-Holland mathematical library. North-Holland, 1982.

\bibitem[DSS96]{DrokinEtAl96}
A.~A. Drokin, A.~V. Shapovalov, and I.~V. Shirokov.
\newblock Algebra of local symmetries of the {S}chr\"{o}dinger equation for the
  hydrogen atom.
\newblock {\em Teoret. Mat. Fiz.}, 106(2):273--284, 1996.

\bibitem[Foc35]{Fock:1935vv}
V.~Fock.
\newblock {Zur Theorie des Wasserstoffatoms}.
\newblock {\em Z. Phys.}, 98:145--154, 1935.

\bibitem[Gon88]{Gonzalez88}
F.~B. Gonzalez.
\newblock Bi-invariant differential operators on the {E}uclidean motion group
  and applications to generalized {R}adon transforms.
\newblock {\em Ark. Mat.}, 26(2):191--204, 1988.

\bibitem[GS90]{GiulleminSternberg90}
V.~Guillemin and S.~Sternberg.
\newblock {\em Variations on a theme by {K}epler}, volume~42 of {\em American
  Mathematical Society Colloquium Publications}.
\newblock American Mathematical Society, Providence, RI, 1990.

\bibitem[Hal13]{Hall2013}
B.~C. Hall.
\newblock {\em Quantum theory for mathematicians}, volume 267 of {\em Graduate
  Texts in Mathematics}.
\newblock Springer, New York, 2013.

\bibitem[Hei43a]{Heisenberg43a}
W.~Heisenberg.
\newblock Die ``beobachtbaren {G}r\"{o}ssen'' in der {T}heorie der
  {E}lementarteilchen.
\newblock {\em Z. Phys.}, 120:513--538, 1943.

\bibitem[Hei43b]{Heisenberg43b}
W.~Heisenberg.
\newblock Die beobachtbaren {G}r\"{o}ssen in der {T}heorie der
  {E}lementarteilchen. {II}.
\newblock {\em Z. Phys.}, 120:673--702, 1943.

\bibitem[Hei44]{Heisenberg44}
W.~Heisenberg.
\newblock Die beobachtbaren {G}r\"{o}ssen in der {T}heorie der
  {E}lementarteilchen. {III}.
\newblock {\em Z. Physik}, 123:93--112, 1944.

\bibitem[HT20]{HigsonTan20}
N.~Higson and Q.~Tan.
\newblock On a spectral theorem of {W}eyl.
\newblock {\em Expo. Math.}, 38(2):180--201, 2020.

\bibitem[Kat76]{Kato76}
T.~Kato.
\newblock {\em Perturbation theory for linear operators}.
\newblock Springer-Verlag, Berlin-New York, second edition, 1976.
\newblock Grundlehren der Mathematischen Wissenschaften, Band 132.

\bibitem[Kna86]{Knapp86}
A.~W. Knapp.
\newblock {\em Representation theory of semisimple groups}, volume~36 of {\em
  Princeton Mathematical Series}.
\newblock Princeton University Press, Princeton, NJ, 1986.
\newblock An overview based on examples.

\bibitem[KO03a]{KobayashiOrstedI}
T.~Kobayashi and B.~Orsted.
\newblock Analysis on the minimal representation of {$O(p,q)$}. {I}.
  {R}ealization via conformal geometry.
\newblock {\em Adv. Math.}, 180(2):486--512, 2003.

\bibitem[KO03b]{KobayashiOrstedII}
T.~Kobayashi and B.~Orsted.
\newblock Analysis on the minimal representation of {$O(p,q)$}. {II}.
  {B}ranching laws.
\newblock {\em Adv. Math.}, 180(2):513--550, 2003.

\bibitem[KO03c]{KobayashiOrstedIII}
T.~Kobayashi and B.~Orsted.
\newblock Analysis on the minimal representation of {$O(p,q)$}. {III}.
  {U}ltrahyperbolic equations on {${\Bbb R}^{p-1,q-1}$}.
\newblock {\em Adv. Math.}, 180(2):551--595, 2003.

\bibitem[Kod49]{Kodaira49}
K.~Kodaira.
\newblock The eigenvalue problem for ordinary differential equations of the
  second order and {H}eisenberg's theory of {$S$}-matrices.
\newblock {\em Amer. J. Math.}, 71:921--945, 1949.

\bibitem[Kos63]{Kostant63}
B.~Kostant.
\newblock Lie group representations on polynomial rings.
\newblock {\em Amer. J. Math.}, 85:327--404, 1963.

\bibitem[KV95]{KnappVogan}
A.~W. Knapp and D.~A. Vogan, Jr.
\newblock {\em Cohomological induction and unitary representations}, volume~45
  of {\em Princeton Mathematical Series}.
\newblock Princeton University Press, Princeton, NJ, 1995.

\bibitem[LL58]{LandauLifshitz}
L.~D. Landau and E.~M. Lifshitz.
\newblock {\em Quantum mechanics: non-relativistic theory. {C}ourse of
  {T}heoretical {P}hysics, {V}ol. 3}.
\newblock Addison-Wesley Series in Advanced Physics. Pergamon Press Ltd.,
  London-Paris; for U.S.A. and Canada: Addison-Wesley Publishing Co., Inc.,
  Reading, Mass;, 1958.
\newblock Translated from the Russian by J. B. Sykes and J. S. Bell.

\bibitem[Mac20]{Maclay2020}
G.~Jordan Maclay.
\newblock Dynamical symmetries of the {H} atom.
\newblock {\em Symmetry}, 12(8), 2020.

\bibitem[Men10]{Meng2010}
G.~Meng.
\newblock Generalized {MICZ}-{K}epler problems and unitary highest weight
  modules, {II}.
\newblock {\em J. Lond. Math. Soc. (2)}, 81(3):663--678, 2010.

\bibitem[MPW13]{Miller13}
W.~Miller, Jr., S.~Post, and P.~Winternitz.
\newblock Classical and quantum superintegrability with applications.
\newblock {\em J. Phys. A}, 46(42):423001, 97, 2013.

\bibitem[MZ11]{MengZhang2011}
G.~Meng and R.~Zhang.
\newblock Generalized {MICZ}-{K}epler problems and unitary highest weight
  modules.
\newblock {\em J. Math. Phys.}, 52(4):042106, 23, 2011.

\bibitem[Nel59]{Nelson59}
E.~Nelson.
\newblock Analytic vectors.
\newblock {\em Ann. of Math. (2)}, 70:572--615, 1959.

\bibitem[Pau26]{Pauli}
W.~Pauli.
\newblock {\"Uber das Wasserstoffspektrum vom Standpunkt der neuen
  Quantenmechanik}.
\newblock {\em Z. Phys.}, 36(5):336--363, 1926.

\bibitem[Sub18]{Subag}
E.~M. Subag.
\newblock Symmetries of the hydrogen atom and algebraic families.
\newblock {\em Journal of Mathematical Physics}, 59(7):071702, 2018.

\bibitem[TdCA98]{Torres1998}
G.~F. Torres~del Castillo and J.L.~Calvario Ac\'{o}cal.
\newblock {On the dynamical symmetry of the quantum Kepler problem}.
\newblock {\em Revista Mexicana de Fisica}, 44(4):344--352, 1998.

\bibitem[Tes12]{Teschl12}
G.~Teschl.
\newblock {\em Ordinary differential equations and dynamical systems}, volume
  140 of {\em Graduate Studies in Mathematics}.
\newblock American Mathematical Society, Providence, RI, 2012.

\bibitem[Wal88]{Wallach}
N.~R. Wallach.
\newblock {\em Real reductive groups. {I}}, volume 132 of {\em Pure and Applied
  Mathematics}.
\newblock Academic Press, Inc., Boston, MA, 1988.

\bibitem[WW96]{WhittakerWatson96}
E.~T. Whittaker and G.~N. Watson.
\newblock {\em A course of modern analysis}.
\newblock Cambridge Mathematical Library. Cambridge University Press,
  Cambridge, 1996.
\newblock An introduction to the general theory of infinite processes and of
  analytic functions; with an account of the principal transcendental
  functions, Reprint of the fourth (1927) edition.

\end{thebibliography}
\bibliographystyle{alpha}

\end{document}